\documentclass[11pt]{article}
\usepackage{mathtools}
\usepackage[T1]{fontenc}
\usepackage{amsfonts}
\usepackage{amsmath}
\usepackage{amssymb}
\usepackage{amsthm}
\usepackage{bbm}
\usepackage{bm}
\usepackage{mathrsfs}
\usepackage{color}
\usepackage{pdfsync}
\usepackage{enumitem}
\usepackage[colorlinks=true, linkcolor=blue, citecolor=blue, urlcolor=blue]{hyperref}
\newcommand{\Banach}{\mathcal{C}_\oplus}
\usepackage{tikz}

\DeclareMathOperator*{\spt}{spt}

\DeclareMathOperator{\proj}{proj}

\newcommand{\RR}{\mathbb{R}}
\newcommand{\R}{\RR}

\newcommand{\NN}{\mathbb{N}}

\newcommand{\eps}{ \varepsilon}

\newcommand{\ball}{\mathcal{B}}
\usepackage[margin=31 mm]{geometry}
\newcommand{\mykill}[1]{}
\usepackage[capitalize, noabbrev]{cleveref}
\crefname{equation}{}{} %

\theoremstyle{plain}
\newtheorem{theorem}{Theorem}[section]
\newtheorem{proposition}[theorem]{Proposition}
\newtheorem{lemma}[theorem]{Lemma}
\newtheorem{corollary}[theorem]{Corollary}
\theoremstyle{definition}

\newtheorem{remark}[theorem]{Remark}

\newtheorem{assumption}[theorem]{Assumption}
\crefname{assumption}{Assumption}{Assumptions}
\Crefname{assumption}{Assumption}{Assumptions}
\theoremstyle{remark}
\AddToHook{env/theorem/begin}{\crefalias{section}{theorem}}
\AddToHook{env/proposition/begin}{\crefalias{theorem}{proposition}}
\AddToHook{env/lemma/begin}{\crefalias{theorem}{lemma}}
\AddToHook{env/corollary/begin}{\crefalias{theorem}{corollary}}
\AddToHook{env/definition/begin}{\crefalias{theorem}{definition}}
\AddToHook{env/remark/begin}{\crefalias{theorem}{remark}}
\AddToHook{env/example/begin}{\crefalias{theorem}{example}}
\AddToHook{env/assumption/begin}{\crefalias{theorem}{assumption}}
\crefname{theorem}{Theorem}{Theorems}
\crefname{proposition}{Proposition}{Propositions}
\crefname{lemma}{Lemma}{Lemmas}
\crefname{corollary}{Corollary}{Corollaries}
\crefname{definition}{Definition}{Definitions}
\crefname{remark}{Remark}{Remarks}
\crefname{example}{Example}{Examples}
\crefname{assumption}{Assumption}{Assumptions}

\newlist{myenum}{enumerate}{3}
\setlist[myenum,1]{label={\rm (H\arabic*)},
                   ref  ={\rm (H\arabic*)}}
\crefname{myenumi}{property}{properties}

\renewenvironment{thebibliography}[1]{%
\begin{oldthebibliography}{#1}%
\setlength{\baselineskip}{.9em}
\linespread{1}
\small
\setlength{\parskip}{.30ex}%
\setlength{\itemsep}{.29em}%
}%
{%
\end{oldthebibliography}%
}

\begin{document}

\title{Linear Convergence of Gradient Descent\\for Quadratically Regularized Optimal Transport}
\date{\today}
\author{  
  Alberto Gonz{\'a}lez-Sanz%
  \thanks{Department of Statistics, Columbia University, ag4855@columbia.edu.} \and  Marcel Nutz%
  \thanks{Departments of Mathematics and Statistics, Columbia University, mnutz@columbia.edu. Research supported by NSF Grants DMS-2106056, DMS-2407074.} \and Andrés Riveros Valdevenito%
  \thanks{Department of Statistics, Columbia University, ar4151@columbia.edu.}
  }
  
\maketitle
\vspace{-1.5em}
\begin{abstract}
In optimal transport, quadratic regularization is an alternative to entropic regularization when sparse couplings or small regularization parameters are desired. Quadratic regularization penalizes transport couplings by the squared $L^2$ norm of their density, or equivalently by the $\chi^2$ divergence. While a number of computational approaches have been shown to work in practice, the dual problem is not strongly convex and theoretical convergence results are scarce. We focus on the dual gradient descent algorithm in a continuous setting and establish linear convergence in \(L^2\), that is, the $L^2$ distance between the iterates and the limiting potentials decreases exponentially fast. The proof is based on a spectral analysis of the linearized gradient descent operator at the optimum. We show that this operator is a strict contraction and that the nonlinear iteration inherits this property after a burn-in period.
\end{abstract}

\vspace{1em}

{\small


}
 \vspace{0em}
\section{Introduction}

\paragraph{Synopsis.} This paper studies the dual problem of quadratically regularized optimal transport and its natural gradient descent iteration.  Unlike entropic regularization, quadratic regularization preserves sparsity of the optimal transport and remains numerically stable for
small regularization parameters. Its analysis, however, is less standard as sparsity entails a lack of strong convexity. Our main result shows linear convergence in \(L^2\) of the gradient iteration for continuous transport problems. Our main mathematical contribution is a spectral analysis of the linearized gradient descent operator at the optimal potentials. It proves that this operator is a strict contraction and implies the linear convergence result. Intuitively, it establishes that the dual objective exhibits nontrivial curvature near the optimum even though global strong convexity fails.

\paragraph{Background.} Optimal transport has become ubiquitous in many areas where distributions or data sets need to be compared, such as statistics, machine learning and image processing. Given compactly supported probability distributions~$P$ and~$Q$ on $\R^d$, the optimal transport problem with quadratic cost is
\begin{equation}\label{otIntro}
  {\rm OT}(P,Q)=  \inf_{\pi\in \Pi( P ,   Q )} \int \frac{1}{2}\|x-y\|^2d\pi(x,y),
\end{equation}
where $\Pi( P ,   Q )$ denotes the set of couplings; that is, probability distributions on $\R^d\times\R^d$ with marginals~$(P,Q)$. The optimal value ${\rm OT}(P,Q)$ defines the Wasserstein distance between~$P$ and~$Q$ and hence is the target of numerous computational approaches (see, e.g.,  \cite{Peyre.Cuturi.2019.Book}). Following~\cite{Cuturi.2013.Neurips}, entropic regularization and the corresponding Sinkhorn algorithm have become standard tools. The entropically regularized  optimal transport problem is
\begin{equation}
    \label{EOT}
    {\rm EOT}_{ \eps}(P,Q)= \inf_{\pi\in \Pi(P, Q)} \int \frac{1}{2}\|x-y\|^2d\pi(x,y)+\eps\,  {\rm KL}\big(\pi\vert P \otimes Q\big),
\end{equation}
where $\eps>0$ is a parameter determining the strength of regularization and ${\rm KL}(\pi\vert P \otimes Q)$ is the Kullback--Leibler divergence between~$\pi$ and the product~$P \otimes Q$. Sinkhorn's iteration can be described as the coordinate-ascent algorithm for the dual problem of~\eqref{EOT} which seeks a pair of functions $f,g:\R^d\to\R$ maximizing
\begin{equation}\label{dual.EOT}
    \int f(x)+ g(y) -\eps e^{\frac{f(x)+g(y)- \frac{1}{2}\|x-y\|^2}{\eps}}  d(P \otimes Q)(x,y) . %
\end{equation}
Thanks to the algebraic properties of the exponential function in~\eqref{dual.EOT}, the coordinate-wise maximization (i.e., optimizing separately $f$ or $g$) has a closed-form solution, leading to the iteration 
\begin{equation}\label{Sinkhorn}
        g_n(y) = - \varepsilon \log\left( \int e^{\frac{f_{n-1}(x) - \frac{1}{2}\|x - y\|^2}{\varepsilon}} \, dP(x) \right), \quad 
        f_n(x) = - \varepsilon \log\left( \int e^{\frac{g_n(y) - \frac{1}{2}\|x - y\|^2}{\varepsilon}} \, dQ(y) \right).
\end{equation}
The exponential also entails that the dual objective~\eqref{dual.EOT} is strongly convex (or strongly concave, to be precise, when restricted to uniformly bounded $f(x)+g(y)$). This implies (see \cite{Carlier.2022.SIOPT}) that the iteration converges linearly  to the maximizer of~\eqref{dual.EOT}, which in turn approximates the dual solution of the optimal transport problem~\eqref{otIntro} in the limit  $\eps\to0$ of vanishing regularization. Several other proofs of linear convergence are known, starting with \cite{FRANKLIN.1989} for the discrete case, and a recent body of literature analyzes the corresponding constants in detail (see~\cite{ChizatDelalandeVaskevicius.25} and the references therein). 

While Sinkhorn's algorithm has been very successful in many applications, it has limitations. Using KL divergence entails that the optimal coupling of~\eqref{EOT} always has full support (equal to the support of $P\otimes Q$). This is known as overspreading in applications, as the true optimal transport for~\eqref{otIntro} is typically sparse (even given by a deterministic map). For example, overspreading can correspond to blurring in an image processing task~\cite{blondel18quadratic} or bias in a manifold learning task~\cite{zhang.2023.manifoldlearningsparseregularised}. Separately, a computational limitation is faced when small regularization parameters~$\eps$ are desired to closely approximate~\eqref{otIntro}: since exponentially large and small values occur in~\eqref{Sinkhorn}, the algorithm tends to become unstable for small values of $\eps$, an issue that can be mitigated only to some extent \cite{Schmitzer.19}. See also \cite{Lahn.Mulchandani.Raghvendra.2019.NEURIPS}, where a Dijkstra-type search algorithm is proposed as a replacement.

\paragraph{QOT.} Starting with \cite{Muzellec.2017.AAAI,blondel18quadratic,EssidSolomon.18}, an alternate approach is to regularize with a different divergence. The most tractable choice is the $\chi^2$ divergence, or equivalently, penalization by the squared $L^2$~norm of the density, leading to the quadratically regularized optimal transport problem,
\begin{equation}\label{qotIntro}
  {\rm QOT}_\eps(P,Q)=  \inf_{\pi\in \Pi( P ,   Q )} \int \frac{1}{2}\|x-y\|^2d\pi(x,y) +\frac{\eps }{2}\left\| \frac{d \pi}{ d( P \otimes   Q )}\right\|^2_{L^{2}( P \otimes   Q )}.
\end{equation}
It is known that ${\rm QOT}_\eps(P,Q)$ approximates the unregularized optimal transport cost ${\rm OT}(P,Q)$ at rate \( \eps^{2/(d+2)} \) as $\eps\to0$ (see \cite{Eckstein.Nutz.2023}, and \cite{GarrizmolinaElAl.2024} for the leading constant). In contrast to entropic regularization, the optimal coupling~$\pi_*$ of~\eqref{qotIntro} has sparse support for small $\eps$; this has been observed empirically since the initial works (e.g., \cite{blondel18quadratic,EssidSolomon.18, Lorenz.2019,BayraktarEckstein.2025.BJ}) and established theoretically more recently in~\cite{WieselXu.24,GonzalezSanzNutz2024.Scalar,gonzalezsanz2026localsparsity}. See \cref{fig:Dim1} for an illustration.
 \begin{figure}[htb]
    \centering
\includegraphics[width=0.51\linewidth]{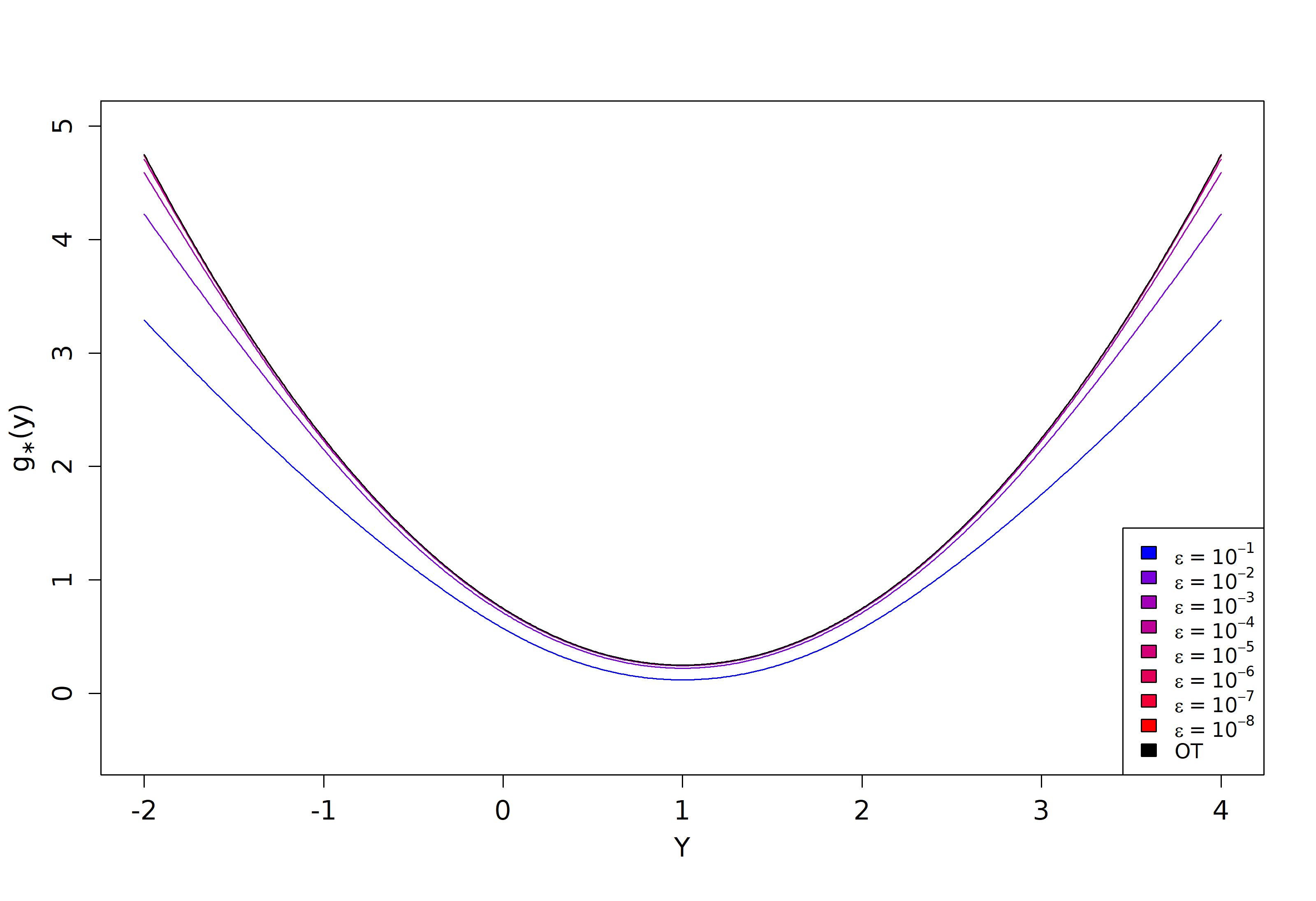}\includegraphics[width=0.51\linewidth]{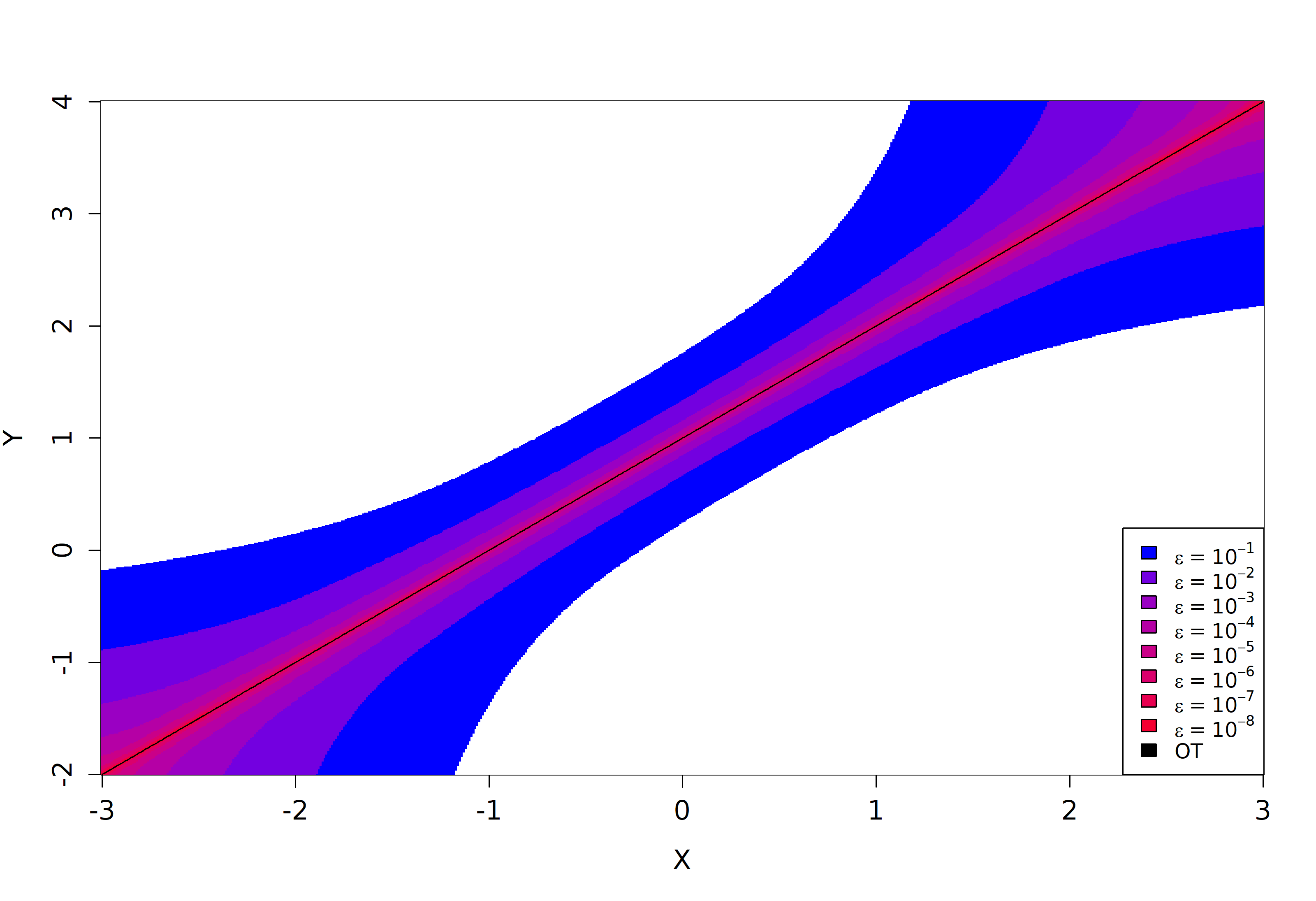}
    \caption{Transporting $P= N(0,1)$ to $Q=N(1,1)$; the measures are truncated to [-3,3] and [-2,4], respectively. Left: Dual solutions of ${\rm QOT}_\eps$ and OT. Right: Supports of the optimal couplings are sparse and converge to the optimal transport map.} 
    \label{fig:Dim1} 
\end{figure}
A number of computational approaches have been considered, mostly targeting the dual problem of~\eqref{qotIntro},
\begin{equation}\label{eq:DualIntro}
    \sup_{f,g: \R^d\to\R} \int f(x) + g(y) -\frac{1}{2\eps} \left(f(x)+ g(y)- \frac{1}{2}\|x-y\|^2\right)_+^{2} d(  P \otimes   Q )(x,y),
\end{equation}
or its first-order condition of optimality,
    \begin{equation}\label{SchrodingerIntro}
        \begin{cases}
            \varepsilon = \displaystyle\int \left( f(x) + g(y) - \tfrac{1}{2} \|x - y\|^2 \right)_+ \, dP(x), \\
            \varepsilon = \displaystyle\int \left( f(x) + g(y) - \tfrac{1}{2} \|x - y\|^2 \right)_+ \, dQ(y).
        \end{cases}
    \end{equation}
Here $(t)_+:=\max\{t,0\}$ denotes the positive part function; it appears as a consequence of the constraint that couplings are nonnegative measures. In the realm of discrete marginals, \cite{Muzellec.2017.AAAI} proposed a mirror gradient method while \cite{EssidSolomon.18} used a Newton-type
algorithm to solve a minimum-cost flow problem on a graph. In~\cite{blondel18quadratic}, the authors leveraged a generic L-BFGS solver and also introduced an alternating minimization scheme. A similar Gauss–Seidel method was suggested in~\cite{Lorenz.2019}; the idea is to alternately solve the equations in~\eqref{SchrodingerIntro}. While Sinkhorn's algorithm implements the analogous iteration for entropic regularization, the equations in~\eqref{SchrodingerIntro} cannot be solved in closed form. To implement this implicit method, \cite{Lorenz.2019} explored both direct search strategies and a semi-smooth Newton method. These approaches were further examined in~\cite{lorenz.2019.preprint} where the authors noted good empirical convergence but a high computational cost per iteration in large-scale settings. To mitigate this, they proposed several alternative methods, including cyclic projections, dual gradient descent, and its accelerated variant. While no theoretical analysis is given, their numerical experiments suggest that all three approaches are efficient and consistent. (Note that the numerical methods and experiments can be found in the preprint~\cite{lorenz.2019.preprint}; they are omitted in the journal version \cite{LorenzMahler.22}.) Several algorithms are also presented in \cite{Pasechnyuk2023Algorithms}, together with bounds for their arithmetic complexity given discrete marginals. In the continuous setting, several works including \cite{EcksteinKupper.21, GeneveyEtAl.16, GulrajaniAhmedArjovskyDumoulinCourville.17, LiGenevayYurochkinSolomon.20, seguy2018large} apply neural networks to the dual problem. For instance, \cite{LiGenevayYurochkinSolomon.20} uses neural networks and gradient descent to compute regularized Wasserstein barycenters.

While numerous algorithms have been used successfully, we are not aware of convergence rates for continuous QOT stated in the literature. (Sublinear rates could, however, be deduced from general optimization results.) Linear convergence is not obvious given the lack of strong convexity in~\eqref{eq:DualIntro} and the limited smoothness of the potentials. 
More generally, conventional wisdom is that quadratically regularized optimal transport ``works'' in practice but is difficult to analyze theoretically. The present work provides not only a linear convergence analysis, but also develops techniques that will be useful in other theoretical studies on regularized optimal transport. 

\paragraph{Contributions.} Specifically, we consider the gradient descent algorithm for the dual problem~\eqref{eq:DualIntro}.
With $\Gamma(f,g)$ denoting the objective function in~\eqref{eq:DualIntro}, the gradient descent (or ascent, to be precise) in $L^2(P)\times L^2(Q)$ with step size $\eta>0$ is 
\begin{equation}\label{GradientDescentIntro}
 \left(  \begin{array}{c}
     f_{n+1}\\
       g_{n+1}
       
\end{array} \right) = \left(  \begin{array}{c}
     f_{n}\\
       g_{n}
     
\end{array} \right)+ \eta\cdot  {\rm D}\Gamma\left(  \begin{array}{c}
     f_{n}\\
       g_{n}
     
\end{array} \right) 
\end{equation}
for $n\geq 0$, where the initial values $(f_0,g_0)$ are given and the explicit form of the gradient is 
\begin{equation}\label{gradientIntro}
{\rm D}\Gamma\left(  \begin{array}{c}
    f\\
    g
\end{array} \right)(x,y)
=  \left(  \begin{array}{c}
     1-\frac{1}{\eps}\int{  \left(f(x)+ g(\tilde y)- \frac{1}{2}\|x-\tilde  y\|^2\right)_+ } dQ(\tilde y)\\
      1-\frac{1}{\eps}\int{  \left(f(\tilde  x)+ g(y)- \frac{1}{2}\|\tilde x-y\|^2\right)_+ } dP(\tilde x)
\end{array} \right).
\end{equation} 
We observe that the iteration~\eqref{GradientDescentIntro} is fairly straightforward to implement, like the Sinkhorn iteration, but it is free from  exponentially large or small values. Indeed, even a naive implementation is stable for very small parameters $\eps$. Evaluating the integrals in~\eqref{gradientIntro} may require replacing~$P$ and~$Q$ by empirical samples, especially in high-dimensional problems. In that respect, it is useful to know that the quadratically regularized optimal transport problem has parametric sample complexity~\cite{GonzalezSanzDelBarrioNutz.25}---it does not suffer from the same curse of dimensionality as optimal transport. 

We provide a theoretical analysis in a general continuous setting with marginals that do not charge small sets and one marginal having connected support. The main result (\cref{Theorem:Explicit}) shows that for step size $\eta<\eps$, the iterates $(f_n,g_n)$ converge linearly in $L^2$ to the solution $(f_*,g_*)$ of the dual problem~\eqref{eq:DualIntro}; that is, there exist $\delta_*<1$ and $n_0\ge0$ such that
\begin{equation}\label{intro-linear-conv}
\| (f_n, g_n) - (f_*, g_*) \|_{L^2(P)\times L^2(Q)} \leq C\delta_*^n
\end{equation}
for all $n\geq n_0$. Our numerical experiments suggest that this bound accurately captures the behavior of the algorithm:  convergence is approximately geometric after a burn-in period. 
Our mathematical analysis centers on the linearization $\mathbb{L}$ of the gradient descent operator $\mathbb{I} + \eta {\rm D}\Gamma$ at the optimum $(f_*,g_*)$. Indeed, we show that the linear operator $\mathbb{L}_n$ mapping $(f_n,g_n)-(f_{*},g_{*})$ to $(f_{n+1},g_{n+1})-(f_{*},g_{*})$ converges in operator norm to $\mathbb{L}$.  The main mathematical contribution of this paper is a spectral analysis of~$\mathbb{L}$. It establishes that the spectrum of the self-adjoint operator $\mathbb{L}$ is contained in~$(-1,1)$ and hence that its operator norm satisfies $\|\mathbb{L}\|_{\mathrm{op}}<1$ (\cref{pr:contractionExplicit}). This norm then gives the eventual contraction constant $\delta_*$ in the main result~\eqref{intro-linear-conv}.

Our proof is tailored to the structure of regularized optimal transport. The gradient descent iteration could also be viewed as a forward--backward splitting method in Hilbert space (with the backward part equal to zero); then, monotone-operator results imply qualitative convergence of the iterates and an \(O(1/n)\) rate for the objective gap \cite{combettes2018monotone}. However, these general arguments do not provide a linear rate.

The basic intuition for our results is geometric: while the dual objective $\Gamma$ fails to be strongly convex due to the positive part function in~\eqref{eq:DualIntro}, it nevertheless exhibits curvature in a suitable neighborhood of the optimizer $(f_{*},g_{*})$. Specifically, the first-order condition~\eqref{SchrodingerIntro} implies that the argument of the positive part function at $(f,g)=(f_{*},g_{*})$ must be positive on a set that cannot be arbitrarily small---while this set may be sparse, it supports the optimal coupling, and that amounts to a lower bound. The spectral analysis of the linearization $\mathbb{L}$ at $(f_{*},g_{*})$ captures this curvature and formalizes the intuition analytically.

We mention that several follow-up works build on the analysis developed in the present paper. In \cite{GonzalezSanzDelBarrioNutz.25}, our analysis is adapted to treat the random operators that arise in the empirical version of QOT based on samples from the marginals. The main result is a central limit theorem for the empirical potentials towards the continuous potentials $(f_{*},g_{*})$. In \cite{GonzalezSanzNutzRiveros.26}, our analysis is refined to show a Polyak--Łojasiewicz (PL) inequality for QOT and better understand the constants that are not explicitly quantified in the present work. In \cite{GonzalezSanzNutz.26stability}, in turn, the curvature around the optimum is similarly exploited to derive local Lipschitz stability of the potentials when the marginals~$(P,Q)$ are perturbed in Wasserstein metric.

\paragraph{Organization.} \Cref{sec:main-results} details the setting, the gradient descent algorithm and the main result on its convergence. \Cref{section-preliminary-result} gathers preliminary results for the proof. \Cref{section:uniform-linearized} studies the linearized gradient descent operator $\mathbb{L}$ and forms the core of our analysis. \Cref{Section:linear-GD} completes the proof of the main result by showing the convergence of $\mathbb{L}_n$ to $\mathbb{L}$. \Cref{section:numerical-examples} concludes with numerical experiments.

\section{Problem statement and main result}\label{sec:main-results}

Let $P,Q$ be probability measures on $\R^d$. The following is a standing assumption throughout the paper.

\begin{assumption}[Marginals]\label{assumption}
    The measures $P,Q\in\mathcal{P}(\R^d)$ do not charge the boundary of any convex subset of $\R^d$. Their topological supports
    $$
      \Omega:=\spt P \qquad\text{and}\qquad \Omega':=\spt Q
    $$
   are compact, and $\Omega'$ is connected.
\end{assumption}

The quadratically regularized optimal transport (QOT) problem with regularization parameter $\eps>0$ is
\begin{equation}\label{qotMain}
  {\rm QOT}_\eps(P,Q):=  \inf_{\pi\in \Pi( P ,   Q )} \int \frac{1}{2}\|x-y\|^2d\pi(x,y) +\frac{\eps }{2}\left\| \frac{d \pi}{ d( P \otimes   Q )}\right\|^2_{L^{2}( P \otimes   Q )}
\end{equation}
with the convention that the last term is $+\infty$ if $\pi \not\ll P\otimes Q$. 
This problem has a unique solution $\pi_*\in \Pi(P,Q)$, and $\pi_*$ is characterized within $\Pi(P,Q)$ by having a density of the form 
\begin{equation}\label{eq:PiQOTrelationDual}
   \frac{d\pi_*}{d(P\otimes Q)}(x,y)= \frac{1}{\eps}\left(f_*(x)+ g_*(y)- \frac{1}{2}\|x-y\|^2\right)_+
\end{equation}
for a pair $(f_*,g_*)\in L^2(P)\times L^2(Q)$ called the potentials. If $(f_*,g_*)$ are potentials, then $(f_*-c,g_*+c)$ are also potentials for any $c\in\R$. To remove this ambiguity, we work with the subspace 
$$
  L_\oplus^2= \left\{(f,g)\in L^2(P)\times L^2(Q): \int fdP=\int gdQ\right\} = \{(c,-c): c\in \R\}^\perp \subset L^2(P)\times L^2(Q).
$$
As a consequence of the connectedness in \cref{assumption}, the potentials are unique in~$L_\oplus^2$ (see \cref{le:Schrodingersystem} below for all these assertions). The orthogonal projection onto $L_\oplus^2$ is 
\begin{align}
        \label{Projection-operator}
    \proj_\oplus :L^2(P)\times L^2(Q)&\to L_\oplus^2, \qquad 
   \left(  \begin{array}{c}
    f\\
    g
\end{array} \right) \mapsto \left(  \begin{array}{c}
    f+ \frac{1}{2} \left(\int g dQ-  \int fdP \right)\\[.2em]
    g-\frac{1}{2} \left(\int g dQ-  \int fdP \right)
\end{array} \right)
\end{align}
and $L_\oplus^2$ is naturally a Hilbert space with the induced inner product   
$$ \langle  (f,g) ,   (u,v)\rangle_{L^2_\oplus} = \langle  (f,g) ,   (u,v)\rangle_{L^2(P)\times L^2(Q)} = \langle f,u\rangle_{L^2(P)} + \langle g,v\rangle_{L^2(Q)}.
$$
The potentials are also characterized as the unique solution of the dual problem of~\eqref{qotMain}, 
\begin{align}\label{eq:dual}
    \sup_{(f,g)\in L_\oplus^2} \Gamma (f,g),
\end{align}
where the dual objective function is
$$ \Gamma (f,g)=  \int f(x) d  P (x)+\int g(y) d  Q (y)-\frac{1}{2\eps}\int{  \left(f(x)+ g(y)- \frac{1}{2}\|x-y\|^2\right)_+^{2} }d(  P \otimes   Q )(x,y).$$
The gradient of $\Gamma$ at $(f,g)\in L_\oplus^2$  is
\begin{equation}\label{eq:gradientGamma}
{\rm D}\Gamma\left(  \begin{array}{c}
    f\\
    g
\end{array} \right)
=  \left(  \begin{array}{c}
     1-\frac{1}{\eps}\int{  \left(f(\cdot)+ g(y)- \frac{1}{2}\|\cdot-y\|^2\right)_+ } dQ(y)\\
      1-\frac{1}{\eps}\int{  \left(f(x)+ g(\cdot)- \frac{1}{2}\|x-\cdot\|^2\right)_+ } dP(x)
     
\end{array} \right) \in L_\oplus^2.
\end{equation}
Thus, the gradient descent algorithm with step size $\eta>0$ is 
\begin{equation}\label{GradientDescent}
 \left(  \begin{array}{c}
     f_{n+1}\\
       g_{n+1}
       
\end{array} \right) = \left(  \begin{array}{c}
     f_{n}\\
       g_{n}
     
\end{array} \right)+ \eta\cdot  {\rm D}\Gamma\left(  \begin{array}{c}
     f_{n}\\
       g_{n}
     
\end{array} \right) 
\end{equation}
for $n\geq 0$, where the initial values $(f_0,g_0)\in L_\oplus^2$ are given inputs. Note that $(f_n,g_n)\in L_\oplus^2$ implies $(f_{n+1},g_{n+1})\in L_\oplus^2$.

\begin{assumption}\label{Assumptions:GD}
The initial values $f_0:\Omega\to\R$ and $g_0:\Omega'\to\R$ are Lipschitz continuous functions normalized such that $\int f_0dP=\int g_0dQ$. The step size $\eta$ satisfies $\eta \in (0, \varepsilon)$.
\end{assumption}

We can now state the main result.

\begin{theorem}[Linear convergence of gradient descent]\label{Theorem:Explicit}
Under \cref{Assumptions:GD}, there exist constants $\delta_* \in (0,1)$ and $n_0 \in \mathbb{N}$ such that the iterates $(f_n, g_n)$ of~\eqref{GradientDescent} satisfy
\[
\| (f_n, g_n) - (f_*, g_*) \|_{L^2_\oplus} \leq \delta_*^n \quad\text{for all $n \geq n_0$.}
\]
\end{theorem}

The proof, which occupies the rest of the paper, has the following structure. We observe that the iterates $(f_n,g_n)$ satisfy 
\begin{align*} 
    \left( \begin{array}{c}
  f_{n+1}-f_{*}    \\
   g_{n+1}-g_{*}  
\end{array}\right)= \mathbb{L}_n\left( \begin{array}{c}
  f_{n}-f_{*}    \\
   g_{n}-g_{*}  
\end{array}\right)
\end{align*}
for an operator $\mathbb{L}_n$ (see \cref{Lemma:representationExplicit}) which converges in operator norm to a limit $\mathbb{L}$ as $n\to\infty$ (\cref{pr:convergenceOperatorsExplicit}). The limiting operator~$\mathbb{L}$ is the linearization of the gradient descent operator at the optimum $(f_*,g_*)$. The key step is to show that the operator norm of $\mathbb{L}$ is strictly smaller than one (\cref{pr:contractionExplicit}). It then follows that $\mathbb{L}_n$ is a strict contraction for $n\geq n_0$, which is the assertion of \cref{Theorem:Explicit}.

\begin{remark}\label{rk:ComparisonWith2overL}
    \Cref{Assumptions:GD} allows for step sizes $\eta < \eps$. A classical result states that for a strongly convex function with $L$-Lipschitz gradient, the gradient descent algorithm converges linearly for any step size $\eta<2/L$ (see \cite[Theorem~2.1.15]{Nesterov.04}). In our problem, the gradient ${\rm D}\Gamma$ is Lipschitz in $L^2_\oplus$ with Lipschitz constant bounded by $2/\eps$ (cf.\ \cref{le:lipschitzgradient}). In that sense, \cref{Assumptions:GD} is in line with the classical result. Our experiments in \cref{section:numerical-examples} suggest that convergence can break down for $\eta \geq \eps$.
\end{remark}

\begin{remark}
    We focus on the transport cost $c(x,y)=\frac{1}{2}\|x-y\|^2$ which is the most important example in practice. While the continuity of this function is used throughout, the specific form is used only to infer the regularity properties in \cref{lemma:QOT-reg}~(i)  which, in turn, are used to argue that certain sets occurring in the proof of \cref{pr:convergenceOperatorsExplicit} are negligible. While we do not know how to rigorously guarantee the latter in general, it seems plausible that \cref{Theorem:Explicit} could extend to more general transport costs. The analysis of the linearized gradient descent operator in \cref{section:uniform-linearized} directly extends to general continuous costs~$c(x,y)$.
\end{remark}

\section{Preliminaries}\label{section-preliminary-result}

Let $\mathcal{C}(\Omega)$ denote the space of continuous functions $f: \Omega\to\R$. For $f\in \mathcal{C}(\Omega)$ and $g\in \mathcal{C}(\Omega')$, we denote $(f\oplus g)(x,y):=f(x)+g(y)$. Consider the quotient space
\[
\Banach := \left( \mathcal{C}(\Omega) \times \mathcal{C}(\Omega') \right) / \sim_{\oplus}
\]
where $(f, g) \sim_{\oplus} (u, v)$ if and only if $f \oplus g = u \oplus v$, and endow $\Banach$ with the norm
\[
\|(u, v)\|_{\mathcal{C}_\oplus} := \inf_{a \in \mathbb{R}} \left\{ \|u + a\|_\infty + \|v - a\|_\infty \right\}.
\]

Next, we detail some properties of the potentials $(f_*,g_*)$ that will be used throughout our convergence analysis. The following system~\eqref{Schrodinger} can be understood as the first-order condition of optimality for the dual problem~\eqref{eq:DualIntro}.

\begin{lemma}\label{le:Schrodingersystem}
There exists a unique pair $(f_*, g_*) \in \Banach$ solving the system
    \begin{equation}\label{Schrodinger}
        \begin{cases}
            \varepsilon = \displaystyle\int \left( f_*(x) + g_*(y) - \tfrac{1}{2} \|x - y\|^2 \right)_+ \, dP(x) & \text{for all } y \in \Omega', \\
            \varepsilon = \displaystyle\int \left( f_*(x) + g_*(y) - \tfrac{1}{2} \|x - y\|^2 \right)_+ \, dQ(y) & \text{for all } x \in \Omega.
        \end{cases}
    \end{equation}
    The pair $(f_*, g_*)$ is also characterized as the unique maximizer of the dual problem~\eqref{eq:DualIntro} and by the relation~\eqref{eq:PiQOTrelationDual} with the primal solution $\pi_*$.
\end{lemma}

\begin{proof}
    Existence of a solution $(f_*, g_*) \in \Banach$ of~\eqref{Schrodinger}, as well as the equivalence of~\eqref{Schrodinger} with solving the dual problem and with~\eqref{eq:PiQOTrelationDual}, are shown for instance in \cite{Nutz.2024}. Uniqueness is also shown in~\cite{Nutz.2024}, but only under additional conditions on $(P,Q)$. Next, we provide a more general proof.
    
    Let $(f_*,g_*)$ and $(f'_*,g'_*)$ be continuous solutions of \eqref{Schrodinger}. Fix $y\in\Omega'$, then \eqref{Schrodinger} implies that 
    $$P\left\{ x \in \Omega : f_*(x) + g_*(y) - \tfrac{1}{2}\|x - y\|^2 > 0 \right\}>0;$$
    in particular, the latter set contains an element $x_0$. Continuity of $g_*$ implies that 
    \begin{align}\label{eq:posNbhd}
    f_*(x_0) + g_*(\tilde y) - \tfrac{1}{2}\|x_0 - \tilde y\|^2 > 0 \quad \mbox{for all $\tilde y$ in a neighborhood $U_y$ of $y$ in $\Omega'$.}
    \end{align}
    The relation \eqref{eq:PiQOTrelationDual} shows that $\frac{d\pi_*}{d(P\otimes Q)}$ admits a continuous version, and it is a general fact that if the density of a measure admits a continuous version, then that version is uniquely determined at every point of the support. Since \eqref{eq:PiQOTrelationDual} holds both with $(f_*,g_*)$ and $(f'_*,g'_*)$, and both are continuous, we conclude from~\eqref{eq:posNbhd} that 
    \begin{align*}
      f_*(x_0) + g_*(\tilde  y) - \tfrac{1}{2}\|x_0 - \tilde  y\|^2 
      =f'_*(x_0) + g'_*(\tilde  y) - \tfrac{1}{2}\|x_0 - \tilde  y\|^2 \quad \mbox{for all~} \tilde y \in U_y.
    \end{align*}
    Hence, $g_*(\tilde y)-g'_*(\tilde y)=f'_*(x_0) - f_*(x_0)=:c$ for all $\tilde y \in U_y$, showing that $g_*-g'_*$ is locally constant in $\Omega'$. As $\Omega'$ is connected, it follows that $g_*-g'_*=c$ is constant in $\Omega'$. It now follows from~\eqref{Schrodinger} that $f_*-f'_*=-c$ (see, e.g., \cite[Lemma~2.4]{Nutz.2024}) and hence $(f_*,g_*)=(f'_*,g'_*)$ in~$\Banach$.
\end{proof}

Next, we detail two properties of the set
\begin{align}\label{eq:defEset}
  \mathcal{E}:=\left\{ (x,y) \in \Omega\times\Omega' : f_*(x) + g_*(y) - \tfrac{1}{2}\|x - y\|^2 \geq 0 \right\}.
\end{align}
We denote its sections by
\begin{align}
\mathcal{S}_x &:= \left\{ y \in \Omega' : f_*(x) + g_*(y) - \tfrac{1}{2}\|x - y\|^2 \geq 0 \right\}, \quad x \in \Omega, \label{eq:xSection}\\
\mathcal{T}_y &:= \left\{ x \in \Omega : f_*(x) + g_*(y) - \tfrac{1}{2}\|x - y\|^2 \geq 0 \right\}, \quad y \in \Omega'. \label{eq:ySection}
\end{align}

\begin{lemma}\label{lemma:QOT-reg} 
\begin{enumerate}
    \item For any $x \in \Omega$, there is a convex set $C_x\subset\R^d$ with nonempty interior such that 
    \begin{align*}
        \mathcal{S}_x = \left\{ y \in \Omega' : f_*(x) + g_*(y) - \tfrac{1}{2}\|x - y\|^2 \geq 0 \right\} &= C_x \cap \Omega', \\
        \mathcal{N}_x := \left\{ y \in \Omega' : f_*(x) + g_*(y) - \tfrac{1}{2}\|x - y\|^2 = 0 \right\} &= \partial C_x \cap \Omega'.
    \end{align*}
    In particular, $\mathcal{N}_x$ is $Q$-negligible. The symmetric statements hold for~$P$.
    \item There exists a constant $\lambda > 0$ such that $Q(\mathcal{S}_x) \geq \lambda$ and $P(\mathcal{T}_y) \geq \lambda$ for all $(x,y) \in \Omega\times\Omega'$.
\end{enumerate}
\end{lemma}

\begin{proof}
  (i) This is due to a concavity property that was previously used in~\cite{GonzalezSanzNutz2024.Scalar,WieselXu.24}; we detail the proof for the sake of completeness. One first observes that $f_*, g_*$ of~\eqref{Schrodinger} can be extended to continuous functions on $\R^d$ such that 
    \begin{equation}\label{SchrodingerRd}
        \begin{cases}
            \varepsilon = \displaystyle\int \left( f_*(x) + g_*(y) - \tfrac{1}{2} \|x - y\|^2 \right)_+ \, dP(x) & \text{for all } y \in \R^d, \\
            \varepsilon = \displaystyle\int \left( f_*(x) + g_*(y) - \tfrac{1}{2} \|x - y\|^2 \right)_+ \, dQ(y) & \text{for all } x \in \R^d.
        \end{cases}
    \end{equation}
    Let $f_*,g_*$ satisfy~\eqref{SchrodingerRd}. Write $F(x)=f_*(x)-\|x\|^2/2$ and $G(y)=g_*(y)-\|y\|^2/2$. We show that $F:\R^d\to\R$ is concave.  Indeed, let $x,x'\in\R^d$ and $\rho\in[0,1]$. Convexity of $t\mapsto (t)_+$ and using~\eqref{SchrodingerRd} at both~$x$ and~$x'$ yield
  \begin{align*}
    &\int \big[\rho F(x) + (1-\rho) F(x') + G(y) + \langle \rho x + (1-\rho)x',y \rangle \big]_+ dQ(y) \\
    &\leq \rho \int \big[ F(x) + G(y) + \langle x,y \rangle \big]_+ dQ(y) + 
    (1-\rho) \int \big[ F(x') + G(y) + \langle x',y \rangle \big]_+ dQ(y) 
    =\eps.
  \end{align*}
  On the other hand, \eqref{SchrodingerRd} at~$x'':=\rho x + (1-\rho)x'$ yields
  \begin{align*}
    \int \big[ F(\rho x + (1-\rho)x') + G(y) + \langle \rho x + (1-\rho)x',y \rangle \big]_+ dQ(y) 
    =\eps.
  \end{align*}
  Together, it follows that $\rho F(x) + (1-\rho) F(x') \leq F(\rho x + (1-\rho)x')$, as claimed.
  
  Analogously, $G$ is concave. In particular, for any $x \in \Omega$, the function
\[
y \mapsto f_*(x) + g_*(y) - \tfrac{1}{2}\|x - y\|^2
\]
is concave, which implies that its super-level set $\hat{\mathcal{S}}_x:=\{y\in\R^d: f_*(x) + g_*(y) - \tfrac{1}{2}\|x - y\|^2\geq0\}$ is convex. Moreover, \eqref{SchrodingerRd} implies that the open set $\{y\in\R^d: f_*(x) + g_*(y) - \tfrac{1}{2}\|x - y\|^2>0\}$ has positive $Q$-measure and in particular is nonempty. As a consequence, $\hat{\mathcal{S}}_x$ is a convex set with nonempty interior and its boundary in $\R^d$ is the zero-level set $\{y\in\R^d: f_*(x) + g_*(y) - \tfrac{1}{2}\|x - y\|^2=0\}$. Since the boundary of a convex set is $Q$-negligible by \cref{assumption}, the proof is complete.

(ii) This follows from an argument given in the proof of \cite[Proposition~5.1]{BayraktarEckstein.2025.BJ} for a class of divergences; however, for the present case, we can also give a straightforward proof: Set
$$
  C := \sup_{x\in\Omega, \, y\in\Omega'}  f_*(x) + g_*(y) - \tfrac{1}{2} \|x - y\|^2.
$$
For any $x\in\Omega$, \eqref{Schrodinger} yields
$
  \varepsilon =  \int \left( f_*(x) + g_*(y) - \tfrac{1}{2} \|x - y\|^2 \right)_+ \, dQ(y) \leq Q(\mathcal{S}_x)C,
$
and now the claim $Q(\mathcal{S}_x) \geq \lambda$ follows with $\lambda := \varepsilon/C$. The proof of $P(\mathcal{T}_y) \geq \lambda$ is analogous.
\end{proof}

\section{Contractivity of linearized gradient descent}  \label{section:uniform-linearized}
In this section, we study the (formal) linearization of the gradient descent operator 
\begin{equation}\label{eq:gradientDescentOperator}
 \left(  \begin{array}{c}
     f\\
       g
\end{array} \right) \mapsto \left(  \begin{array}{c}
     f\\
       g
\end{array} \right)+ \eta\cdot  {\rm D}\Gamma\left(  \begin{array}{c}
     f\\
       g
\end{array} \right) 
\end{equation}
at the dual optimum $(f_*,g_*)$; namely, the operator $\mathbb{L}:L^2_\oplus\to L^2_\oplus$, 
\begin{align}\label{eq:defL}
\mathbb{L} \begin{pmatrix}
    f \\
    g
\end{pmatrix}
= \proj_\oplus \begin{pmatrix}
    f \left(1 - \frac{\eta}{\varepsilon} Q(\mathcal{S}_{(\cdot)}) \right) - \frac{\eta}{\varepsilon} \int_{\mathcal{S}_{(\cdot)}} g \, dQ \\
    g \left(1 - \frac{\eta}{\varepsilon} P(\mathcal{T}_{(\cdot)}) \right) - \frac{\eta}{\varepsilon} \int_{\mathcal{T}_{(\cdot)}} f \, dP
\end{pmatrix},
\end{align}
where $x\mapsto\mathcal{S}_{x}$ and $y\mapsto \mathcal{T}_{y}$ were defined in \cref{eq:xSection,eq:ySection}.\footnote{The expression in~\eqref{eq:defL} serves as the definition of~$\mathbb{L}$. While it is indeed the Gâteaux derivative, Fr\'echet differentiability of~\eqref{eq:gradientDescentOperator} as an operator $L^2_\oplus\to L^2_\oplus$ need not hold.} 
We will show that \( \mathbb{L} \) is self-adjoint on the Hilbert space \( (L^2_\oplus, \langle \cdot, \cdot \rangle_{L^2_\oplus}) \), so that its operator norm can be computed via
\[
\|\mathbb{L}\|_{\mathrm{op}} = \sup_{\|(f,g)\|_{L^2_\oplus} \leq 1} \left| \left\langle \mathbb{L}(f,g), (f,g) \right\rangle \right|.
\]

The main result of this section is the following.

\begin{proposition}\label{pr:contractionExplicit}
    Under \cref{Assumptions:GD}, the operator $\mathbb{L}:L^2_\oplus\to L^2_\oplus$ is a strict contraction; i.e., $ \|\mathbb{L}\|_{\mathrm{op}}<1$.
\end{proposition} 

The proof, which occupies the rest of the section, has the following structure. Since \( \mathbb{L} \) is self-adjoint, either \(\alpha:= \|\mathbb{L}\|_{\mathrm{op}} \) or  \( \alpha:=-\|\mathbb{L}\|_{\mathrm{op}} \) belongs to the spectrum of \( \mathbb{L} \),  and 
\begin{equation*}
    \alpha= \sup_{\|(f,g)\|_{L^2_\oplus}= 1}\langle \mathbb{L}(f,g),(f,g) \rangle_{L^2_\oplus}
    \quad \text{or}\quad 
    \alpha= \inf_{\|(f,g)\|_{L^2_\oplus}= 1}\langle \mathbb{L}(f,g),(f,g) \rangle_{L^2_\oplus}.
\end{equation*}
We deduce that
$$ \left\|\left(\begin{array}{c}
     f_n ((1-\alpha)-   \frac{\eta}{\eps}Q(\mathcal{S}_{(\cdot)})) -\frac{\eta}{\eps}\int_{\mathcal{S}_{(\cdot)}} g_n dQ +b_n\\
     g_n ((1-\alpha) -   \frac{\eta}{\eps} P(\mathcal{T}_{(\cdot)})) -\frac{\eta}{\eps} \int_{\mathcal{T}_{(\cdot)}} f_n dP -b_n
\end{array} \right) \right\|_{L^2(P)\times L^2(Q)}\to 0 $$
for a sequence $(f_n,g_n)\in L^2_\oplus$ with unit norm. Moreover, we establish that the sequence 
$$ \left(\begin{array}{c}
      \int_{\mathcal{S}_{(\cdot)}} g_n dQ \\
     \int_{\mathcal{T}_{(\cdot)}} f_n dP 
\end{array} \right) \in L^2(P)\times L^2(Q)$$
is strongly pre-compact. We deduce that  $(f_n,g_n)_n$ is strongly convergent along a subsequence, and then, that one of the numbers \(\pm \|\mathbb{L}\|_{\mathrm{op}} \) is an eigenvalue of $\mathbb{L}$. We conclude by proving that all eigenvalues of \( \mathbb{L} \) lie in the open interval \( (-1,1) \), which is the main part of the argument. 

We start with several auxiliary results and detail the proof of \cref{pr:contractionExplicit} at the end of the section. Denote by $\ball_r(x)$ the open ball of radius $r$ around $x$.

\begin{lemma}\label{lemma:compact} 
    The following operators are compact:
    \begin{align*}
    \mathbb{A}_1 : L^2(Q)\to L^2(P), & \qquad \mathbb{A}_1( g )= \frac{\int_{\mathcal{S}_{(\cdot)} }g dQ}{Q(\mathcal{S}_{(\cdot)} )}, \\[.2em]
    \mathbb{A}_2 : L^2(P)\to L^2(Q), & \qquad \mathbb{A}_2(f)=\frac{\int_{\mathcal{T}_{(\cdot)} }f dP}{P(\mathcal{T}_{(\cdot)} )}.
    \end{align*}
\end{lemma}

\begin{proof}
We prove that \( \mathbb{A}_1 \) is compact, the second claim is analogous. It suffices to show that if \( \{u_n\}_n \subset L^2(Q) \) converges weakly to zero, then \( \| \mathbb{A}_1(u_n) \|_{L^2(P)} \to 0 \). In view of \cref{lemma:QOT-reg}~(ii),
\[
\| \mathbb{A}_1(u_n) \|_{L^2(P)}^2 = \int \left( \frac{\int_{\mathcal{S}_x} u_n(y) \, dQ(y)}{Q(\mathcal{S}_x)} \right)^2 dP(x) \leq \lambda^{-2} \int \left( \int_{\mathcal{S}_x} u_n(y) \, dQ(y) \right)^2 dP(x).
\]
Set \( h_n(x) := \left( \int_{\mathcal{S}_x} u_n(y) \, dQ(y) \right)^2 \). As \( u_n \to 0 \) weakly in \( L^2(Q) \),
\[
h_n(x) = \left\langle \mathbb{I}_{\mathcal{S}_x}, u_n \right\rangle_{L^2(Q)}^2 \to 0
\quad \text{for all~} x \in \Omega.
\]
On the other hand,  Jensen's inequality yields 
$
h_n(x)  \leq \|u_n\|_{L^2(Q)}^2.
$
As weakly convergent sequences are norm-bounded, this shows that $(h_n)$ is uniformly bounded by a constant. Using the dominated convergence theorem, we conclude that \( \| \mathbb{A}_1(u_n) \|_{L^2(P)} \to 0 \).
\end{proof}

\begin{lemma}\label{lemma:compact2} For $\mathbb{A}_1, \mathbb{A}_2$ as defined in \cref{lemma:compact}, the equation 
\begin{equation}
    \label{eq:Fredholm}
    \left( \begin{array}{c}
  \mathbb{A}_1( g )\\
  \mathbb{A}_2( f )
\end{array}\right) = - \left( \begin{array}{c}
 f\\
 g
\end{array}\right)
\end{equation}
on $L^2(P)\times L^2(Q)$ has the solution set $\{(f,g)=(c,-c): c\in\R\}$. 
\end{lemma}

\begin{proof}
Let $(f,g)\in L^2(P)\times L^2(Q)$ be any solution of \eqref{eq:Fredholm}, and recall the definition~\eqref{eq:defEset}. Using the first row of \eqref{eq:Fredholm} yields
$$\int_{\mathcal{E}}  (f(x)   + g(y))  h(x) dQ(y) dP(x) =\int_{\Omega}\left(\int_{\mathcal{S}_x}  (f(x)   + g(y))  h(x) dQ(y) \right)dP(x)  = 0$$
for all $h\in L^2(P)$. In particular, choosing $h:=f$ yields
$$ 0= \int_\mathcal{E} (f^2(x)+ f(x) g(y))   dQ(y) dP(x).$$
Analogously, 
$$ \int_\mathcal{E} (g^2(y)+ f(x) g(y))   dQ(y) dP(x) =0.$$
Adding the two displays shows that
    $\int_\mathcal{E} (f(x)+g(y))^2   dQ(y) dP(x)=0$
and hence that $f\oplus g =0$ holds $P\otimes Q$-a.s.~in $\mathcal{E}$.

Next, we argue that $f$ and $g$ admit continuous versions. Recall~\eqref{eq:PiQOTrelationDual} and that $f_*,g_*$ are continuous (see \cite[Lemma 2.6]{Nutz.2024}). In particular, 
$\pi_*$ is supported in $\mathcal{E}$ and its density is continuous. For any test function $h\in L^2(P)$, it follows from $f\oplus g =0$ $P\otimes Q$-a.s.~in $\mathcal{E}$ that 
\begin{align*}
  0&= \int  h(x) (f(x)+g(y) )\frac{d\pi_*}{d(P\otimes Q)}(x,y)   dP(x) dQ(y) \\
  &= \int \left( h(x) f(x)+  h(x)\int   \frac{d\pi_*}{d(P\otimes Q)}(x,y)  g(y) dQ(y)  \right) dP(x).
\end{align*}
Therefore,
\begin{equation}
    \label{eq:system-regularized}
    f(x)=- \int  \frac{d\pi_*}{d(P\otimes Q)}(x,y) g(y) dQ(y)\quad\mbox{$P$-a.s.},
\end{equation}
and as the right-hand side is continuous in~$x$, this shows that $f$ admits a continuous version. The argument for $g$ is analogous. We henceforth replace $f$ and $g$ by these continuous versions.

We now upgrade the almost-sure identity $f\oplus g =0$ to a pointwise identity on 
\[
    \mathcal E^+
    :=
    \left\{
    (x,y)\in\Omega\times\Omega':
    \xi_*(x,y)>0
    \right\}, \quad \xi_*(x,y):=f_*(x)+g_*(y)-\frac12\|x-y\|^2;
\]
i.e., we claim that
\begin{equation}\label{eq:fg-pointwise-Eplus}
    f(x)+g(y)=0
    \qquad\text{for all }(x,y)\in\mathcal E^+ .
\end{equation}
Indeed, fix $(x_0,y_0)\in\mathcal E^+$. Suppose for contradiction that
$f(x_0)+g(y_0)\neq0$. Since $f,g,\xi_*$ 
are continuous, there exist relative neighborhoods $V\subset\Omega$ of $x_0$
and $U\subset\Omega'$ of $y_0$, and a constant $\delta>0$, such that
\[
    V\times U\subset \mathcal E^+,
    \qquad
    |f(x)+g(y)|\geq \delta
    \quad\text{for all }(x,y)\in V\times U .
\]
As $x_0\in\spt P$ and $y_0\in\spt Q$, we have $P(V)>0$ and $Q(U)>0$. Hence,
using $\mathcal E^+\subset\mathcal E$ and the identity
$f\oplus g=0$ $P\otimes Q$-a.s. on $\mathcal E$, we obtain
\[
    0
    =
    \int_{\mathcal E} (f(x)+g(y))^2\,dP(x)dQ(y)
    \geq
    \int_{V\times U} (f(x)+g(y))^2\,dP(x)dQ(y)
    \geq
    \delta^2 P(V)Q(U)>0,
\]
a contradiction. This proves \eqref{eq:fg-pointwise-Eplus}.

It remains to show that the continuous versions $f,g$ with $f\oplus g=0$ on $\mathcal{E}^+$ satisfy $(f,g)=(-c,c)$ on $\Omega\times\Omega'$, for some $c\in\R$. Let $y\in\Omega'$. As in~\eqref{Schrodinger}, there exist $x_0\in\Omega$ with $\xi_*(x_0,y)>0$ and a neighborhood $U_y$ of $y$ in $\Omega'$ such that $\{x_0\}\times U_y\subset\mathcal{E}^+$. Thus $g(\tilde y)=-f(x_0)$ for all $\tilde y\in U_y$. We have shown that $g$ is locally constant. As $\Omega'$ is connected by \cref{assumption}, it follows that $g\equiv c$ is constant on $\Omega'$, and now~\eqref{eq:Fredholm} implies that $f\equiv -c$ on $\Omega$.
\end{proof}

\begin{lemma}\label{le:LselfAdjoint} The operator $\mathbb{L}$ is self-adjoint in $(L_\oplus^2,\langle  \cdot ,   \cdot\rangle_{L^2_\oplus} )$.
\end{lemma}

\begin{proof}
We first consider the auxiliary operator $\mathbb{M}: L^2(P)\times L^2(Q)$ defined by
\[
\mathbb{M} \begin{pmatrix}
    f \\
    g
\end{pmatrix}
= \begin{pmatrix}
    f \cdot Q(\mathcal{S}_{(\cdot)}) + \displaystyle\int_{\mathcal{S}_{(\cdot)}} g \, dQ \\
    g \cdot P(\mathcal{T}_{(\cdot)}) + \displaystyle\int_{\mathcal{T}_{(\cdot)}} f \, dP
\end{pmatrix}.
\]
The following representation readily yields that $\mathbb{M}$ is self-adjoint:
$$\mathbb{M}\left(\begin{array}{c}
     f  \\
     g 
\end{array} \right)(x,y)=\left(\begin{array}{c}
   \int \mathbb{I}_{\mathcal{E}}(x, y') (f(x)  + g(y')) dQ(y') \\
     \int \mathbb{I}_{\mathcal{E}}(x', y) (f(x')  + g(y)) dP(x')
\end{array} \right),   $$
where     
$$ \mathcal{E}=\left\{(x,y) \in \Omega\times \Omega':  f_*(x)+g_*(y)\geq \frac{1}{2} \|x-y\|^2\right\}.$$
As a consequence, $\mathbb{M}$ induces a self-adjoint operator \( \mathbb{M}_\oplus := \proj_\oplus \mathbb{M} \) on \( L^2_\oplus \). Indeed, for every $(f,g),(u,v)\in L^2_\oplus$ we have
\begin{align*}
    \langle (f,g), \proj_\oplus \mathbb{M}(u,v) \rangle_{L^2_\oplus}&=\langle (f,g), \mathbb{M}(u,v) \rangle_{L^2(P)\times L^2(Q)}\\
    &=\langle  \mathbb{M}(f,g) ,(u,v)\rangle_{L^2(P)\times L^2(Q)} = \langle  \proj_\oplus \mathbb{M}(f,g) ,(u,v)\rangle_{L^2_\oplus}.
\end{align*}
Recalling from~\eqref{eq:defL} that $\mathbb{L}=\mathbb{I}-\frac{\eta}{\varepsilon}\mathbb{M}_\oplus$, it follows that $\mathbb{L}$ is also self-adjoint.
\end{proof}

\begin{proof}[Proof of \cref{pr:contractionExplicit}]
To show that $ \|\mathbb{L}\|_{\mathrm{op}}<1$, recall (e.g., \cite[Proposition~6.9]{Brezis2011}) that \cref{le:LselfAdjoint} implies $\|\mathbb{L}\|_{\mathrm{op}}=\max\{|\alpha^+|,|\alpha^-|\}$ where
\begin{equation*}
    \alpha^+:= \sup_{\|(f,g)\|_{L^2_\oplus}= 1}\langle \mathbb{L}(f,g),(f,g) \rangle_{L^2_\oplus},
    \qquad 
    \alpha^-:= \inf_{\|(f,g)\|_{L^2_\oplus}= 1}\langle \mathbb{L}(f,g),(f,g) \rangle_{L^2_\oplus} .
\end{equation*}
We set
\begin{equation}\label{eq:alphaCases}
    \alpha := 
    \begin{cases}
    \alpha^+, & \mbox{if } \|\mathbb{L}\|_{\mathrm{op}}=\alpha^+, \\
    \alpha^-, & \mbox{otherwise}.
    \end{cases}
\end{equation}
Suppose for contradiction that $\|\mathbb{L}\|_{\mathrm{op}}\geq 1$. Note that this implies  $\alpha\geq1$ in the first case of~\eqref{eq:alphaCases} and $\alpha\leq -1$ in the second.

By the definition of $\alpha$, there exists a sequence $(f_n,g_n)_n$ with $\|(f_n,g_n)\|_{L^2_\oplus}=1$ such that $\langle \mathbb{L}(f_n,g_n),(f_n,g_n) \rangle_{L^2_\oplus} \to \alpha$. Recall that $(f_n,g_n)$ can be considered as elements of $L^2(P)\times L^2(Q)$ with $\int f_n dP= \int g_n dQ$. Note that $\|\mathbb{L}\|_{\mathrm{op}}=|\alpha|$ implies $\|\mathbb{L} (f_n,g_n)\|^2_{L^2_\oplus}\leq\alpha^2$ and hence
\begin{align*}
    \|(\mathbb{L}- \alpha \mathbb{I})(f_n,g_n)  \|^2_{L^2_\oplus}&= -2\alpha \langle\mathbb{L} (f_n,g_n),  (f_n,g_n) \rangle_{L^2_\oplus}+\alpha^2+ \|\mathbb{L} (f_n,g_n)\|^2_{L^2_\oplus} \nonumber \\
    &\leq 2\alpha(\alpha- \langle\mathbb{L} (f_n,g_n),  (f_n,g_n) \rangle_{L^2_\oplus}) \to 0.
\end{align*}
Hence, there exists a sequence $\{b_n\}_n \subset\R$ such that 
\begin{align}\label{eq:limitLexplicit}
\left\|\left(\begin{array}{c}
     f_n ((1-\alpha)-   \frac{\eta}{\eps}Q(\mathcal{S}_{(\cdot)})) -\frac{\eta}{\eps}\int_{\mathcal{S}_{(\cdot)}} g_n dQ +b_n\\
     g_n ((1-\alpha) -   \frac{\eta}{\eps} P(\mathcal{T}_{(\cdot)})) -\frac{\eta}{\eps} \int_{\mathcal{T}_{(\cdot)}} f_n dP -b_n
\end{array} \right) \right\|_{L^2(P)\times L^2(Q)}\to 0.
\end{align}
In fact, the sequence $\{b_n\}_n$ is bounded by the choice of $(f_n,g_n)$, hence converges to a limit~$b$ after passing to a subsequence. After passing to another subsequence, the Banach--Alaoglu theorem yields that the bounded sequence $(f_n,g_n)_n$ has a weak limit $(f,g)$ in $L^2(P)\times L^2(Q)$. Recalling \cref{lemma:compact} and the fact that compact operators map weakly convergent to strongly convergent sequences, it follows that 
$$ \left\| \left(\begin{array}{c}
     \mathbb{A}_1 (g_n)  \\
    \mathbb{A}_2 (f_n) 
\end{array} \right)- \left(\begin{array}{c}
     \mathbb{A}_1 (g)  \\
    \mathbb{A}_2 (f) 
\end{array} \right)   \right\|_{L^2(P)\times L^2(Q)}\to 0 .$$
Together with~\eqref{eq:limitLexplicit},
we deduce that
$$ \left(\begin{array}{c}
     f_n \\
     g_n
\end{array} \right)\to \left(\begin{array}{c}
     \frac{ \frac{\eta}{\eps}\int_{\mathcal{S}_{(\cdot)}} g dQ -b}{(1-\alpha)-   \frac{\eta}{\eps} Q(\mathcal{S}_{(\cdot)})}\\[.5em]
     \frac{ \frac{\eta}{\eps}\int_{\mathcal{T}_{(\cdot)}} f dP +b}{(1-\alpha) -   \frac{\eta}{\eps} P(\mathcal{T}_{(\cdot)}) }
\end{array} \right) \quad\mbox{in }L^2(P)\times L^2(Q),$$
where we have used that the denominator is bounded away from zero thanks to  $|\alpha|\geq 1$ and  $\frac{\eta}{\eps} P(\mathcal{T}_{(\cdot)}), \frac{\eta}{\eps}Q(\mathcal{S}_{(\cdot)})\in [\frac{\lambda\,\eta}{ \eps},1],$ 
where $\lambda$ is as in \cref{lemma:QOT-reg}~(ii). In particular, the sequence $(f_n,g_n)$ converges strongly. It follows that $(f_n,g_n)$ converges strongly to its weak limit $(f,g)$, and in particular that $\|(f,g)\|_{L^2_\oplus}=1$ and $\int f dP= \int g dQ$. Thus, the equation 
$$ \left(\begin{array}{c}
     f \\
     g
\end{array} \right)= \left(\begin{array}{c}
     \frac{ \frac{\eta}{\eps}\int_{\mathcal{S}_{(\cdot)}} g dQ -b}{(1-\alpha)-   \frac{\eta}{\eps} Q(\mathcal{S}_{(\cdot)})}\\[.5em]
     \frac{ \frac{\eta}{\eps}\int_{\mathcal{T}_{(\cdot)}} f dP +b}{(1-\alpha) -   \frac{\eta}{\eps} P(\mathcal{T}_{(\cdot)}) }
\end{array} \right) $$
admits the solution $(f,g)\in L^2_\oplus$,
or
\begin{equation}\label{eq:stepWithb}
    \begin{cases}
    f ({(1-\alpha)-   \frac{\eta}{\eps} Q(\mathcal{S}_{(\cdot)})})= \frac{\eta}{\eps}\int_{\mathcal{S}_{(\cdot)}} g dQ -b,\\
    g({(1-\alpha) -   \frac{\eta}{\eps} P(\mathcal{T}_{(\cdot)}) })=\frac{\eta}{\eps} \int_{\mathcal{T}_{(\cdot)}} f dP+b.
\end{cases}
\end{equation}
Next, we show that $b=0$. Integrating the first and second equation of~\eqref{eq:stepWithb} with respect to~\( P \) and~\( Q \), respectively, and applying Fubini's theorem, we obtain
\begin{align}
(1-\alpha)\int f(x) \, dP(x) & = \frac{\eta}{\eps}\int_{\mathcal{E}} (f(x) + g(y)) \, d(P \otimes Q)(x,y) - b, \label{eq:interStep1}\\
(1-\alpha)\int g(y) \, dQ(y) &= \frac{\eta}{\eps}\int_{\mathcal{E}} (f(x) + g(y)) \, d(P \otimes Q)(x,y) + b.\label{eq:interStep2}
\end{align}
Subtracting \eqref{eq:interStep2} from \eqref{eq:interStep1}, we find
$
(1-\alpha) \left(\int f \, dP - \int g \, dQ \right) = -2b,
$
and recalling that $\int f dP= \int g dQ$, we conclude \( b = 0 \).

In summary, \( (f,g) \) satisfies $\|(f,g)\|_{L^2_\oplus}=1$ and $\int f dP= \int g dQ$ and solves the system
\begin{align}\label{eq:systemForContrad}
\begin{cases}
f \left((1-\alpha) - \dfrac{\eta}{\varepsilon} Q(\mathcal{S}_{(\cdot)})\right) = \dfrac{\eta}{\varepsilon} \int_{\mathcal{S}_{(\cdot)}} g \, dQ, \\[0.8em]
g \left((1-\alpha) - \dfrac{\eta}{\varepsilon} P(\mathcal{T}_{(\cdot)})\right) = \dfrac{\eta}{\varepsilon} \int_{\mathcal{T}_{(\cdot)}} f \, dP.
\end{cases}
\end{align}

\emph{Case $\alpha >1$ or $\alpha <-1$.} Then, \eqref{eq:systemForContrad} implies that $f$ is bounded $P$-a.s.\ and 
$$ \|f\|_\infty \leq  \begin{cases}
     \|f\|_\infty\sup_{x,y}\frac{\frac{\eta}{\eps} Q(\mathcal{S}_{x}) }{|1-\alpha|+ \frac{\eta}{\eps}Q(\mathcal{S}_{x})  } \frac{ \frac{\eta}{\eps} P(\mathcal{T}_{y}) }{|1-\alpha|+ \frac{\eta}{\eps} P(\mathcal{T}_{y})}  & {\rm if }\ \alpha>1,\\[.5em]
       \|f\|_\infty \sup_{x,y}\frac{ \frac{\eta}{\eps}Q(\mathcal{S}_{x}) }{1-\alpha - \frac{\eta}{\eps}Q(\mathcal{S}_{x}) } \frac{ \frac{\eta}{\eps}P(\mathcal{T}_{y}) }{1-\alpha- \frac{\eta}{\eps}P(\mathcal{T}_{y})} & {\rm if }\ \alpha<-1.
\end{cases}$$
We observe that the above supremum is $<1$ in either case, and thus that $f=0$. Similarly, $g=0$, contradicting that $\|(f,g)\|_{L^2_\oplus}=1$. 

\emph{Case $\alpha =1$.} In this case, \eqref{eq:systemForContrad} specializes to
\begin{equation*}
    \label{case-alpha=1}
    \begin{cases}
    f Q(\mathcal{S}_{(\cdot)})=-\int_{\mathcal{S}_{(\cdot)}} g dQ, \\
    g   P(\mathcal{T}_{(\cdot)} )=-\int_{\mathcal{T}_{(\cdot)}} f dP,
\end{cases}
\end{equation*}
which by \cref{lemma:compact2} means that $(f,g)=(0,0)$ in $L^2_\oplus$, again contradicting  $\|(f,g)\|_{L^2_\oplus}=1$. 

\emph{Case $\alpha =-1$.} In this case, \eqref{eq:systemForContrad} can be written as
\begin{equation*}%
\begin{cases}
    2f(x)=\frac{\eta}{\eps}\int_{\mathcal{S}_{x}} (f(x)+g(y')) dQ(y'), \\
    2g(y)=\frac{\eta}{\eps}\int_{\mathcal{T}_{y}} (f(x')+g(y)) dP(x').
\end{cases}
\end{equation*}
Adding these equations, taking squares, applying the inequality $(a+b)^2\leq 2(a^2+b^2)$ and then Jensen's inequality, we deduce
\begin{align*}
   4\left(  f(x)+g(y) \right)^2 &= \left(\frac{\eta}{\eps} \right)^2  \left(\int_{\mathcal{S}_{x} } (f(x)+g(y')) d Q(y') + \int_{\mathcal{T}_{y} } (f(x')+g(y)) d P(x') \right)^2\\
   &\leq 2\left(\frac{\eta}{\eps} \right)^2  \left(\left(\int_{\mathcal{S}_{x} } (f(x)+g(y')) d Q(y')\right)^2 + \left(\int_{\mathcal{T}_{y} } (f(x')+g(y)) d P(x') \right)^2 \right)\\
   & \leq 2\left(\frac{\eta}{\eps} \right)^2  \left( Q(\mathcal{S}_{x}) \int  (f(x)+g(y'))^2 d Q(y')+ P(\mathcal{T}_{y} )\int   (f(x')+g(y))^2 d P(x')  \right).
\end{align*}
Now integrating w.r.t.\ $P\otimes Q$ yields
\begin{align*}
   0<\int \left(  f(x)+g(y) \right)^2 dP(x)d Q(y)
   & \leq \left(\frac{\eta}{\eps} \right)^2  \int (f(x)+g(y))^2 dP(x) d Q(y) \\
   & <  \int (f(x)+g(y))^2 dP(x) d Q(y)
\end{align*}
as $\eta<\eps$ by \cref{Assumptions:GD}, and we obtain the desired contradiction to $\|(f,g)\|_{L^2_\oplus}=1$.
\end{proof}

\section{Proof of linear convergence} \label{Section:linear-GD}

Recall from \eqref{GradientDescent} that the gradient descent iterates satisfy
\begin{align}\label{eq:algodescription}
    \left(  \begin{array}{c}
     f_{n+1}-f_{n}\\
       g_{n+1}-g_n    
    \end{array} \right) = \frac{\eta}{\eps}\cdot \proj_\oplus\left(  \begin{array}{c}
     \eps-\int{  \left(f_n(\cdot)+ g_n(y)- \frac{1}{2}\|\cdot-y\|^2\right)_+ } dQ(y)\\
      \eps-\int{  \left(f_n(x)+ g_n(\cdot)- \frac{1}{2}\|x-\cdot\|^2\right)_+ } dP(x)    
\end{array} \right).
\end{align}
We first represent the iterates in a form convenient for our analysis, with the operator~$\mathbb{L}_n$ introduced in \cref{Lemma:representationExplicit}. To prove the main result, it then remains to show that $\mathbb{L}_{n}$ converges to $\mathbb{L}$ in operator norm. We first show that $(f_n, g_n)\to (f_*, g_*)$; this is a straightforward Arzel\`a--Ascoli argument (\cref{lemma:consistency}). Next, the proof of $\mathbb{L}_{n}\to \mathbb{L}$ is given in \cref{pr:convergenceOperatorsExplicit}. Combining  $\|\mathbb{L}_{n}- \mathbb{L}\|_{\mathrm{op}}\to0$ with the fact that $\mathbb{L}$ is a contraction (\cref{pr:contractionExplicit}), we complete the proof of the main result in \cref{co:mainRes}.

Let $\mathcal{L}_1$  denote one-dimensional Lebesgue measure.

\begin{lemma}\label{Lemma:representationExplicit}
    The gradient descent iterates $(f_n,g_n)$ satisfy 
\begin{align}\label{eq:GDwrittenWithLn}
    \left( \begin{array}{c}
  f_{n+1}-f_{*}    \\
   g_{n+1}-g_{*}  
\end{array}\right)= \mathbb{L}_n\left( \begin{array}{c}
  f_{n}-f_{*}    \\
   g_{n}-g_{*}  
\end{array}\right)
\end{align}
for the operator 
\begin{align}\label{eq:defLn}
     \mathbb{L}_n\left( \begin{array}{c}
f    \\
   g  
\end{array}\right):= \proj_\oplus\left( \begin{array}{c}
    f(1- \frac{\eta}{\eps}\cdot[\mathcal{L}_1\otimes Q](\mathcal{S}_{n,(\cdot)}))-   \frac{\eta}{\eps}\cdot\int_{\mathcal{S}_{n,(\cdot)}} g(y) d[\mathcal{L}_1\otimes  Q](\lambda, y) \\
        g(1-\frac{\eta}{\eps}\cdot [\mathcal{L}_1\otimes P](\mathcal{T}_{n,(\cdot)}))-   \frac{\eta}{\eps}\cdot\int_{\mathcal{T}_{n,(\cdot)}} f(x) d[\mathcal{L}_1\otimes  P](\lambda, x)
    \end{array} \right),
\end{align}
where
\begin{align*}
     \mathcal{S}_{n,x}&:=\left\{(\lambda,y)\in [0,1]\times \Omega':
      \ \lambda(f_{*}(x)+g_{*}(y))+ (1-\lambda) (f_{n}(x)+g_{n}(y)) \geq \frac{1}{2}\|x-y\|^2\right\}, \\
    \mathcal{T}_{n,y}&:=\left\{(\lambda,x)\in [0,1]\times \Omega:
      \ \lambda(f_{*}(x)+g_{*}(y))+ (1-\lambda) (f_{n}(x)+g_{n}(y)) \geq \frac{1}{2}\|x-y\|^2\right\} .
\end{align*}
\end{lemma}

\begin{proof}
    In view of \eqref{Schrodinger}, \eqref{eq:algodescription} implies that
\begin{multline}\label{eq:algodescriptionIncrement}
    \left( \begin{array}{c}
        {f}_{n+1}-{f}_{*}  \\
        {g}_{n+1}-{g}_{*} 
    \end{array} \right)- \left( \begin{array}{c}
        {f}_{n}-{f}_{*}  \\
        {g}_{n}-{g}_{*}  
    \end{array} \right)\\
    = \frac{\eta}{\eps}\cdot \proj_\oplus\left( \begin{array}{c}
     \int \left(f_{*}(\cdot)+g_{*}(y)-\frac{1}{2}\|\cdot-y\|^2\right)_+-  \left(f_{n}(\cdot)+g_{n}(y)-\frac{1}{2}\|\cdot-y\|^2\right)_+ dQ(y) \\
        \int \left(f_{*}(x)+g_{*}(\cdot)-\frac{1}{2}\|x-\cdot\|^2\right)_+-  \left(f_{n}(x)+g_{n}(\cdot)-\frac{1}{2}\|x-\cdot\|^2\right)_+ dP(x)
    \end{array} \right).
\end{multline}
After applying the fundamental theorem of calculus in the form
\begin{align*}
\phi(1)_+ - \phi(0)_+ &= \int_0^1 \frac{d}{d\lambda} [\phi(\lambda)_+] \, d\mathcal{L}_1(\lambda)
\end{align*}
to the functions
\begin{align*}
\phi(\lambda) &= \lambda(f_*(x) + g_*(y)) + (1 - \lambda)(f_n(x) + g_n(y)) -\frac{1}{2}\|x-y\|^2,\\
\frac{d}{d\lambda} [\phi(\lambda)_+] &= \mathbb{I}_{\phi(\lambda) \ge 0} \cdot (f_*(x) + g_*(y) - f_n(x) - g_n(y)) \quad\mbox{a.e.},
\end{align*}
we get 
\begin{multline*}
    \left( \begin{array}{c}
        {f}_{n+1}-{f}_{*}  \\
        {g}_{n+1}-{g}_{*}  
    \end{array} \right)- \left( \begin{array}{c}
        {f}_{n}-{f}_{*}  \\
        {g}_{n}-{g}_{*}  
    \end{array} \right)\\
    = - \frac{\eta}{\eps}\cdot\proj_\oplus\left( \begin{array}{c}
    (f_{n}-f_{*}) [\mathcal{L}_1\otimes Q](\mathcal{S}_{n,(\cdot)})+   \int_{\mathcal{S}_{n,(\cdot)}} (g_{n}(y)-g_{*}(y)) d[\mathcal{L}_1\otimes  Q](\lambda, y) \\
        (g_{n}-g_{*}) [\mathcal{L}_1\otimes P](\mathcal{T}_{n,(\cdot)})+   \int_{\mathcal{T}_{n,(\cdot)}} (f_{n}(x)-f_{*}(x)) d[\mathcal{L}_1\otimes  P](\lambda, x)
    \end{array} \right)
  \end{multline*}
and the claim follows.
\end{proof}

The next two lemmas establish that the gradient descent iterates are uniformly bounded and equicontinuous.

\begin{lemma}\label{Lemma:representationExplicit2}
Let $\eta\in (0, \eps]$, then
\begin{equation}
    \label{Non-uniform-contraction}
    \left\| \left( \begin{array}{c}
  f_{n+1}-f_{*}    \\
   g_{n+1}-g_{*}  
\end{array}\right) \right\|_{\mathcal{C}_\oplus} \leq  2\cdot \left\| \left( \begin{array}{c}
  f_{0}-f_{*}    \\
   g_{0}-g_{*}  
\end{array}\right) \right\|_{\mathcal{C}_\oplus} .
\end{equation}
\end{lemma}

\begin{proof}
    Up to taking equivalence classes in $\mathcal{C}_\oplus$, \eqref{eq:algodescriptionIncrement} states that
    \begin{align*}
       & \left( \begin{array}{c}
  f_{n+1}-f_{*}    \\
   g_{n+1}-g_{*}  
\end{array}\right)= \left( \begin{array}{c}
  f_{n}-f_{*}    \\
   g_{n}-g_{*}  
\end{array}\right) \\
&\quad + \frac{\eta}{\eps}\cdot \left(  \begin{array}{c}
     \int{  \left(f_*(\cdot)+ g_*(y)- \frac{1}{2}\|\cdot-y\|^2\right)_+ } -   \left(f_n(\cdot)+ g_n(y)- \frac{1}{2}\|\cdot-y\|^2\right)_+  dQ(y)\\
      \int  \left(f_*(x)+ g_*(\cdot)- \frac{1}{2}\|x-\cdot\|^2\right)_+  -   \left(f_n(x)+ g_n(\cdot)- \frac{1}{2}\|x-\cdot\|^2\right)_+  dP(x)
\end{array} \right).
    \end{align*}
Using the inequality
\begin{equation}\label{ineq-positive-part}
    (t)_+ -  (s)_+ \leq \mathbb{I}_{\{t\geq 0\}} (t-s) 
\end{equation}
with $t=f_*(x)+ g_*(y)- \frac{1}{2}\|x-y\|^2$ and $s=f_n(x)+ g_n(y)- \frac{1}{2}\|x-y\|^2$, we infer
 \begin{align*}
    {f}_{n+1}(x)-{f}_{*}(x)
    &\leq ({f}_{n}(x)-{f}_{*}(x))  +\frac{\eta}{\eps}\cdot\int_{\mathcal{S}_{x}} (f_{*}(x)-f_{n}(x))+  (g_{*}(y)-g_{n}(y)) dQ(y)\\
    &=  ({f}_{n}(x)-{f}_{*}(x))\left(1  -\frac{\eta}{\eps}\cdot Q(\mathcal{S}_{x})\right)- \frac{\eta}{\eps}\cdot \int_{\mathcal{S}_{x}} (g_{n}(y)-g_{*}(y)) dQ(y)\\
    &\leq \|{f}_{n}-{f}_{*}\|_\infty\left(1  -\frac{\eta}{\eps}\cdot Q(\mathcal{S}_{x})\right)+ \frac{\eta}{\eps} \cdot Q(\mathcal{S}_{x}) \|g_{n}-g_{*}\|_\infty.%
\end{align*}
Note that the right-hand side is a convex combination of $\|{f}_{n}-{f}_{*}\|_\infty$ and $\|{g}_{n}-{g}_{*}\|_\infty$ as $\eta\in (0,\eps]$. 
Writing
\begin{equation}
    \label{eq:Sn}
    \tilde{\mathcal{S}}_{n,x}:=  \left\{y\in\Omega': \ f_{n}(x)+ g_{n}(y)\geq \frac{1}{2}\|x-y\|^2 \right\}, \quad x\in \Omega, \ n\in \NN,
\end{equation}
we analogously get
 \begin{align*}
    {f}_{*}(x)-{f}_{n+1}(x)
    &\leq \|{f}_{n}-{f}_{*}\|_\infty\left(1  -\frac{\eta}{\eps}\cdot Q( \tilde{\mathcal{S}}_{n,x})\right)+ \frac{\eta}{\eps} \cdot Q( \tilde{\mathcal{S}}_{n,x}) \|g_{n}-g_{*}\|_\infty,
\end{align*}
a convex combination of the same quantities. 
Together, it follows that 
$$ \|{f}_{n+1}-{f}_{*}\|_\infty\leq  \max\left(\|{f}_{n}-{f}_{*}\|_\infty,  \|g_{n}-g_{*}\|_\infty\right) .$$
Repeating the same argument for $\|{g}_{n+1}-{g}_{*}\|_\infty$, we conclude 
$$  \max\left(\|{f}_{n+1}-{f}_{*}\|_\infty,  \|g_{n+1}-g_{*}\|_\infty\right) \leq  \max\left(\|{f}_{n}-{f}_{*}\|_\infty,  \|g_{n}-g_{*}\|_\infty\right)$$
for all $n$ and hence 
$$  \max\left(\|{f}_{n+1}-{f}_{*}\|_\infty,  \|g_{n+1}-g_{*}\|_\infty\right) \leq  \max\left(\|{f}_{0}-{f}_{*}\|_\infty,  \|g_{0}-g_{*}\|_\infty\right).$$
The claim follows after recalling the definition of the norm $\|\cdot\|_{\mathcal{C}_\oplus}$.
\end{proof}

\begin{lemma}\label{lemma:modulus-cont}
    Let $\eta\in (0, \eps]$ and let $(f_0,g_0)$ be Lipschitz with constant  $L_0$. Then for every $n\geq1$, the gradient descent iterates $ (f_n,g_n)$  are Lipschitz with constant $L$, where
    $$L=\max\{L_0,C\}, \qquad C:=\max\{\|x-y\|: x\in\Omega, y\in\Omega' \}.$$ 
\end{lemma}
\begin{proof}
    Arguing inductively, let $L_{n-1}$ denote the Lipschitz constant of $f_{n-1}$. 
    Fix $x,x' \in \Omega$ with $x\neq x' $. By \eqref{eq:algodescription},
        \begin{align*}
    {f}_{n}(x)-{f}_{n}(x' )= ({f}_{n-1}(x)-{f}_{n-1}(x')) & -\frac{\eta}{\eps}\cdot\int \left(f_{n-1}(x)+g_{n-1}(y)-\frac{1}{2}\|x-y\|^2\right)_+ dQ(y)\\
    &+\frac{\eta}{\eps}\cdot\int \left(f_{n-1}(x' )+g_{n-1}(y)-\frac{1}{2}\|x' -y\|^2\right)_+ dQ(y).
\end{align*}
Recalling the definition of $\tilde{\mathcal{S}}_{n-1,x'}$ from \cref{eq:Sn} and using the inequality \eqref{ineq-positive-part}, we get   
 \begin{align*}
    &{f}_{n}(x)-{f}_{n}(x' )\\
    &\leq ({f}_{n-1}(x)-{f}_{n-1}(x'))  +\frac{\eta}{\eps}\cdot\int_{\tilde{\mathcal{S}}_{n-1,x'}} (f_{n-1}(x')-f_{n-1}(x))+\frac{\|x-y\|^2-\|x' -y\|^2}{2} dQ(y)\\
    &\leq  ({f}_{n-1}(x)-{f}_{n-1}(x'))(1  -\frac{\eta}{\eps}\cdot Q(\tilde{\mathcal{S}}_{n-1,x'}))+ \frac{\eta}{\eps}\cdot\int_{\tilde{\mathcal{S}}_{n-1,x'}}\frac{\|x-y\|^2-\|x' -y\|^2}{2} dQ(y)
    \\&\leq ({f}_{n-1}(x)-{f}_{n-1}(x'))(1  -\frac{\eta}{\eps}\cdot Q(\tilde{\mathcal{S}}_{n-1,x'}))+ \frac{\eta}{\eps}\cdot C\|x-x'\|Q(\tilde{\mathcal{S}}_{n-1,x'}),
\end{align*}
where $C=\max\{\|x-y\|: x\in\Omega, y\in\Omega' \}$.
As a consequence, 
$$  \frac{{f}_{n}(x)-{f}_{n}(x' )}{\|x-x'\|}  \leq L_{n-1} (1  -\frac{\eta}{\eps}\cdot Q(\tilde{\mathcal{S}}_{n-1,x'}))+ \frac{\eta}{\eps}\cdot Q(\tilde{\mathcal{S}}_{n-1,x'}) C.$$
Noting that the right-hand side is a convex combination of  $L_{n-1}$ and $C$, we have 
$$ \frac{{f}_{n}(x)-{f}_{n}(x' )}{\|x-x'\|}  \leq \max\left( L_{n-1} , C \right)$$
and the claim for $f_n$ follows. The proof for $g_n$ is analogous.
\end{proof}

\begin{lemma}\label{le:lipschitzgradient} The operator ${\rm D} \Gamma: L^{2}_{\oplus} \to L^{2}_{\oplus}$, given by  \cref{eq:gradientGamma}, is $2/\eps$-Lipschitz.
\end{lemma}
\begin{proof}
    Denote the projections of ${\rm D} \Gamma$ to each of its coordinates by
    \begin{align*}
        {\rm D} \Gamma_{1}(f,g) & :=  1-\frac{1}{\eps}\int{ \left(f(\cdot)+ g(y)- \frac{1}{2}\|\cdot-y\|^2\right)_+ } dQ(y), \\
        {\rm D} \Gamma_{2}(f,g) & :=  1-\frac{1}{\eps}\int{  \left(f(x)+ g(\cdot)- \frac{1}{2}\|x-\cdot\|^2\right)_+ } dP(x).
    \end{align*}
    For $(f,g)$ and $(u,v)$ in $L^{2}_{\oplus}$,
    \begin{multline*}
        |{\rm D}\Gamma_{1}(f,g) - {\rm D}\Gamma_{1}(u,v)| \\ =
        \frac{1}{\eps} \left| \int \left(u(\cdot)+ v(y)- \frac{1}{2}\|\cdot-y\|^2\right)_+ - \left(f(\cdot)+ g(y)- \frac{1}{2}\|\cdot-y\|^2\right)_+ dQ(y) \right|,
    \end{multline*}
    and applying the inequality $|(t)_{+} - (s)_{+}| \leq |t - s|$ yields
    \begin{align*}
        |{\rm D}\Gamma_{1}(f,g) - {\rm D}\Gamma_{1}(u,v)| & \leq \frac{1}{\eps} \int |u(\cdot) + v(y) - f(\cdot) - g(y)| dQ(y) \\
        & \leq \frac{1}{\eps} \left( |u(\cdot) - f(\cdot)| + \int |v(y) - g(y)| dQ(y) \right).
    \end{align*}
    Taking the square and integrating, we use Jensen's inequality and $(a+b)^{2} \leq 2a^{2} + 2b^{2}$ to get
    \begin{align*}
        \| {\rm D}\Gamma_{1}(f,g) - {\rm D}\Gamma_{1}(u,v)\|^{2}_{L^{2}(P)}  & \leq \frac{2}{\eps^{2}} \left( \|f - u\|^{2}_{L^{2}(P)} + \|g - v\|^{2}_{L^{2}(Q)} \right) \\
        & = \frac{2}{\eps^{2}} \|(f,g) - (u,v)\|^{2}_{L^{2}_{\oplus}}.
    \end{align*}
    Using the analogous bound for ${\rm D}\Gamma_{2}$, we obtain
    \begin{align*}
        \| {\rm D}\Gamma(f,g) - {\rm D}\Gamma(u,v)\|^{2}_{L^{2}_{\oplus}} & = \| {\rm D}\Gamma_{1}(f,g) - {\rm D}\Gamma_{1}(u,v)\|^{2}_{L^{2}(P)} + \| {\rm D}\Gamma_{2}(f,g) - {\rm D}\Gamma_{2}(u,v)\|^{2}_{L^{2}(Q)} \\
        & \leq \frac{4}{\eps^{2}}  \|(f,g) - (u,v)\|^{2}_{L^{2}_{\oplus}},
    \end{align*}
    as desired. 
\end{proof}

We can now conclude the convergence of $ (f_n, g_n)$ to $(f_*, g_*)$ in $\Banach$.

\begin{lemma}\label{lemma:consistency}
Let $\eta\in (0, \eps)$, then
    $\| (f_n, g_n)-(f_*, g_*)\|_{\Banach}\to 0.$ 
\end{lemma}
\begin{proof}
  In view of \cref{Lemma:representationExplicit2,lemma:modulus-cont}, the Arzel\`a--Ascoli theorem shows that there exists a subsequence $(f_{n_{k}}, g_{n_{k}})$ converging to a limit $(f_\infty, g_\infty)$ in $\mathcal{C}_\oplus$. 
  
  Recalling \cref{le:lipschitzgradient}, ${\rm D}\Gamma$ is Lipschitz with constant $2/\eps$, which by a classical result of convex optimization (see line~6 in the proof of \cite[Theorem~2.1.14]{Nesterov.04}) implies
  \begin{align*}
       \tilde\eta \,\| {\rm D}\Gamma(f_{n},g_{n}) \|^{2}_{L^{2}_\oplus} \leq  \Gamma(f_{n+1}, g_{n+1}) - \Gamma(f_{n},g_{n})
  \end{align*}
  for $\tilde\eta:=\eta \left( 1 - \frac{\eta}{\eps} \right)>0$.
  Summing over $0 \leq n \leq N$, we deduce for any $N\in\mathbb{N}$ that
  \begin{align*}
      \sum_{n=0}^{N} \| {\rm D}\Gamma(f_{n},g_{n}) \|^{2}_{L^{2}_\oplus} \leq \frac{1}{\tilde\eta} (\Gamma(f_{N+1}, g_{N+1}) - \Gamma(f_{0},g_{0})) \leq \frac{1}{\tilde{\eta}} (\Gamma(f_*, g_*) - \Gamma(f_{0},g_{0})) < \infty. 
  \end{align*}
  It follows that ${\rm D}\Gamma(f_{n},g_{n}) \to 0$ in $L^{2}_{\oplus}$. By continuity of ${\rm D}\Gamma$, this implies ${\rm D}\Gamma(f_{\infty},g_{\infty}) = 0$, meaning that $(f_\infty, g_\infty)$ solves \cref{Schrodinger}. By \cref{le:Schrodingersystem}, $(f_*,g_*)$ is the unique solution of \cref{Schrodinger}, and the claim follows.
\end{proof}

We can now prove the main technical result of this section, namely the convergence $\mathbb{L}_{n} \to\mathbb{L}$ as operators on $L^{2}_{\oplus}$.

\begin{proposition}\label{pr:convergenceOperatorsExplicit}
Let \cref{Assumptions:GD} hold. We have
 $\|\mathbb{L}_{n} -\mathbb{L}\|_{\mathrm{op}}\to 0.$
\end{proposition}

\begin{proof}
     Comparing the definitions of $\mathbb{L}$ and $\mathbb{L}_n$ in~\eqref{eq:defL} and~\eqref{eq:defLn}, we see that $\|\mathbb{L}_{n} -\mathbb{L}\|_{\mathrm{op}}\to 0$ is implied by the two limits 
    \begin{equation}\label{eq:LinftyP-first-term-Ln}
 \left\|
 [\mathcal{L}_1\otimes Q](\mathcal{S}_{n,(\cdot)})
 -
 Q(\mathcal{S}_{(\cdot)})
 \right\|_{L^\infty(P)}
 \to 0,
\end{equation}
and
\begin{equation}\label{eq:LinftyP-second-term-Ln}
 \sup_{\|h\|_{L^2(Q)}\leq 1}
 \left\|
 \int_{\mathcal{S}_{n,(\cdot)}} h(y)\,
 d[\mathcal{L}_1\otimes Q](\lambda,y)
 -
 \int_{\mathcal{S}_{(\cdot)}} h(y)\,dQ(y)
 \right\|_{L^\infty(P)}
 \to 0,
\end{equation}
and the symmetric results for the second component. Clearly \cref{eq:LinftyP-second-term-Ln} implies \cref{eq:LinftyP-first-term-Ln} by specializing to $h\equiv1$, hence we focus on \cref{eq:LinftyP-second-term-Ln}. 
Write
\[
 \xi_*(x,y)
 :=
 f_*(x)+g_*(y)-\frac12\|x-y\|^2,
 \qquad
 \xi_n(x,y)
 :=
 f_n(x)+g_n(y)-\frac12\|x-y\|^2 .
\]
Set
\[
 r_n
 :=
 \sup_{x\in\Omega,\ y\in\Omega'}
 |\xi_n(x,y)-\xi_*(x,y)|.
\]
We first note that \(r_n\to0\). Indeed, for every \(a\in\mathbb R\),
\[
\begin{aligned}
 |(f_n-f_*)(x)+(g_n-g_*)(y)|
 &\leq
 \|f_n-f_*+a\|_\infty+\|g_n-g_*-a\|_\infty .
\end{aligned}
\]
Taking infimum over \(a\) and then supremum over
\((x,y)\) gives
$
 r_n
 \leq
 \|(f_n,g_n)-(f_*,g_*)\|_{\mathcal C_\oplus},
$
which tends to zero by \cref{lemma:consistency}.

For \(x\in\Omega\) and \(y\in\Omega'\), define
\[
 m_n(x,y)
 :=
 \mathcal{L}_1
 \left\{
 \lambda\in[0,1]:
 \lambda \xi_*(x,y)+(1-\lambda)\xi_n(x,y)\geq0
 \right\},
 \qquad
 m(x,y):=\mathbb{I}_{\{\xi_*(x,y)\geq0\}} .
\]
By Fubini's theorem,
\[
 [\mathcal{L}_1\otimes Q](\mathcal{S}_{n,x})
 =
 \int m_n(x,y)\,dQ(y),
 \qquad
 Q(\mathcal{S}_x)
 =
 \int m(x,y)\,dQ(y),
\]
and, for \(h\in L^2(Q)\),
\[
 \int_{\mathcal{S}_{n,x}} h(y)\,
 d[\mathcal{L}_1\otimes Q](\lambda,y)
 =
 \int h(y)m_n(x,y)\,dQ(y),
 \quad
 \int_{\mathcal{S}_x}h(y)\,dQ(y)
 =
 \int h(y)m(x,y)\,dQ(y).
\]

Next, we show that 
\begin{equation}\label{eq:mn-m-bound}
 |m_n(x,y)-m(x,y)|
 \leq
 \mathbb{I}_{\{|\xi_*(x,y)|\leq r_n\}} .
\end{equation}
Indeed, if \(\xi_*(x,y)>r_n\), then \(\xi_n(x,y)>0\), and hence
\[
 \lambda \xi_*(x,y)+(1-\lambda)\xi_n(x,y)>0
 \qquad\text{for all }\lambda\in[0,1].
\]
Thus \(m_n(x,y)=m(x,y)=1\). Similarly, if \(\xi_*(x,y)<-r_n\), then
\(\xi_n(x,y)<0\), so \(m_n(x,y)=m(x,y)=0\). Therefore the two kernels can differ only on
\(\{|\xi_*|\leq r_n\}\), proving \eqref{eq:mn-m-bound}.

We next prove the uniform boundary-layer estimate
\begin{equation}\label{eq:uniform-boundary-layer-Q}
 \omega_Q(r)
 :=
 \sup_{x\in\Omega}
 Q\{y\in\Omega':|\xi_*(x,y)|\leq r\}
 \to0
 \qquad\text{as }r\downarrow0 .
\end{equation}
Suppose not. Then there exist \(\varepsilon_0>0\), \(r_k\downarrow0\), and
\(x_k\in\Omega\) such that
\[
 Q\{y\in\Omega':|\xi_*(x_k,y)|\leq r_k\}\geq \varepsilon_0
 \qquad\text{for all }k.
\]
By compactness of \(\Omega\), after passing to a subsequence we may assume that
\(x_k\to x\in\Omega\). Since \(\xi_*\) is continuous on the compact set
\(\Omega\times\Omega'\), it is uniformly continuous, and therefore
\[
 \delta_k
 :=
 \sup_{y\in\Omega'}
 |\xi_*(x_k,y)-\xi_*(x,y)|
 \to0 .
\]
Hence
\[
 \{y\in\Omega':|\xi_*(x_k,y)|\leq r_k\}
 \subset
 \{y\in\Omega':|\xi_*(x,y)|\leq r_k+\delta_k\}.
\]
It follows that, for every \(\delta>0\) and all sufficiently large \(k\),
\[
 Q\{y\in\Omega':|\xi_*(x,y)|\leq\delta\}\geq \varepsilon_0 .
\]
Letting \(\delta\downarrow0\) and using continuity from above of \(Q\), we obtain
\[
 Q\{y\in\Omega':\xi_*(x,y)=0\}\geq \varepsilon_0,
\]
which contradicts \cref{lemma:QOT-reg}~(i). This proves
\eqref{eq:uniform-boundary-layer-Q}.

Let \(h\in L^2(Q)\) with \(\|h\|_{L^2(Q)}\leq1\). For every
\(x\in\Omega\), the Cauchy--Schwarz inequality and~\eqref{eq:mn-m-bound} yield
\[
\begin{aligned}
&\left|
 \int_{\mathcal{S}_{n,x}} h(y)\,
 d[\mathcal{L}_1\otimes Q](\lambda,y)
 -
 \int_{\mathcal{S}_x} h(y)\,dQ(y)
 \right|  \\
&\qquad =
 \left|
 \int h(y)(m_n(x,y)-m(x,y))\,dQ(y)
 \right| \\
&\qquad \leq
 \|h\|_{L^2(Q)}
 \left(
 \int |m_n(x,y)-m(x,y)|^2\,dQ(y)
 \right)^{1/2} \\
&\qquad \leq
 Q\{y\in\Omega':|\xi_*(x,y)|\leq r_n\}^{1/2}
 \leq
 \omega_Q(r_n)^{1/2}.
\end{aligned}
\]
Taking first the \(P\)-essential supremum in \(x\), and then the supremum over
\(\|h\|_{L^2(Q)}\leq1\), gives
\[
 \sup_{\|h\|_{L^2(Q)}\leq 1}
 \left\|
 \int_{\mathcal{S}_{n,(\cdot)}} h(y)\,
 d[\mathcal{L}_1\otimes Q](\lambda,y)
 -
 \int_{\mathcal{S}_{(\cdot)}} h(y)\,dQ(y)
 \right\|_{L^\infty(P)}
 \leq
 \omega_Q(r_n)^{1/2}
 \to0. \qedhere
\]
\end{proof}

\begin{corollary}\label{co:mainRes}
Let \cref{Assumptions:GD} hold and  $\delta_*>\|\mathbb{L} \|_{\mathrm{op}}$. Then there exists $n_0\in  \mathbb{N}$ such that $\|\mathbb{L}_n \|_{\mathrm{op}} \leq \delta_*$ for all $n\geq n_0$.
\end{corollary}

Taken together, \cref{pr:contractionExplicit,co:mainRes} show that $\mathbb{L}_n$ is a uniform contraction for $n\geq n_0$. In view of~\eqref{eq:GDwrittenWithLn}, it follows that $\|(f_n,g_n)-(f_*,g_*)\|\leq C\delta_*^n$, which is  \cref{Theorem:Explicit} after possibly redefining $n_0,\delta_*$.

\section{Numerical experiments}\label{section:numerical-examples}

In this section, we provide numerical experiments for the gradient ascent algorithm. The key quantity of interest is
$$ \Delta_n := \| (f_{n+1}, g_{n+1}) - (f_{n}, g_{n}) \|_{L^2_\oplus}.$$
On the one hand, we see from~\eqref{eq:algodescription} that
\begin{align*}
    \Delta_n = \frac{\eta}{\eps} \left\|\left(  \begin{array}{c}
     \eps-\int{  \left(f_n(\cdot)+ g_n(y)- \frac{1}{2}\|\cdot-y\|^2\right)_+ } dQ(y)\\
      \eps-\int{  \left(f_n(x)+ g_n(\cdot)- \frac{1}{2}\|x-\cdot\|^2\right)_+ } dP(x)    
\end{array} \right)\right\|_{L^2_\oplus}
\end{align*}
and hence $\Delta_n$ has a direct interpretation as $\eta/\eps$ times the $L^2$ norm of the constraint violation in~\eqref{Schrodinger}. This captures how far the measure with density $\frac{1}{\eps} \left(f_n(x)+ g_n(y)- \frac{1}{2}\|x-y\|^2\right)_+$ is from being a coupling of~$P$ and~$Q$. On the other hand, we will observe that $\Delta_n = O(\delta_*^n)$ for large $n$, for some  $\delta_*\in(0,1)$, implying that also $\| (f_n, g_n) - (f_*, g_*) \|_{L^2_\oplus} = O(\delta_*^n)$. This will confirm that \cref{Theorem:Explicit} accurately captures the convergence behavior.

\begin{figure}
    \centering
    \includegraphics[width=0.49\linewidth]{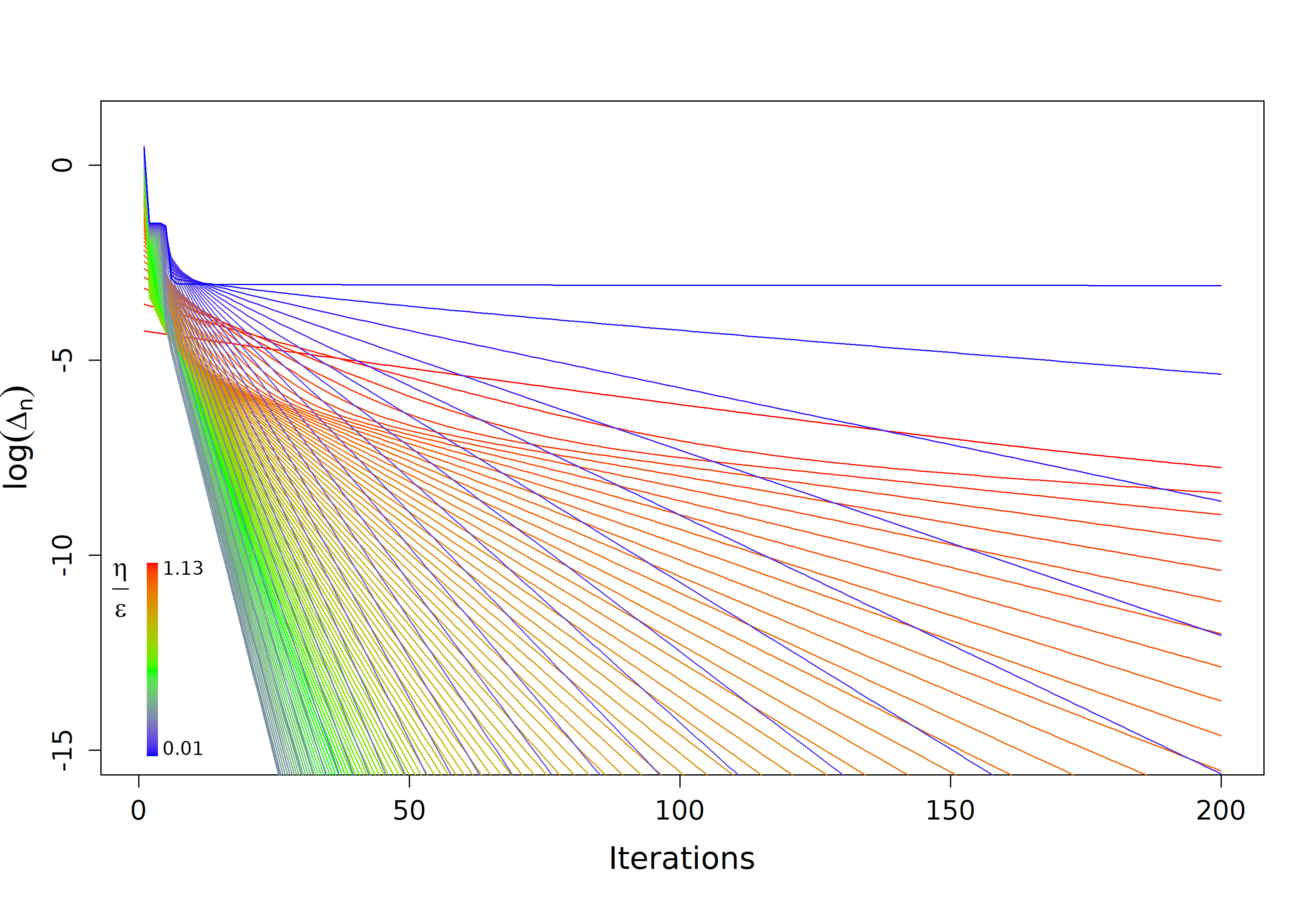}
    \includegraphics[width=0.49\linewidth]{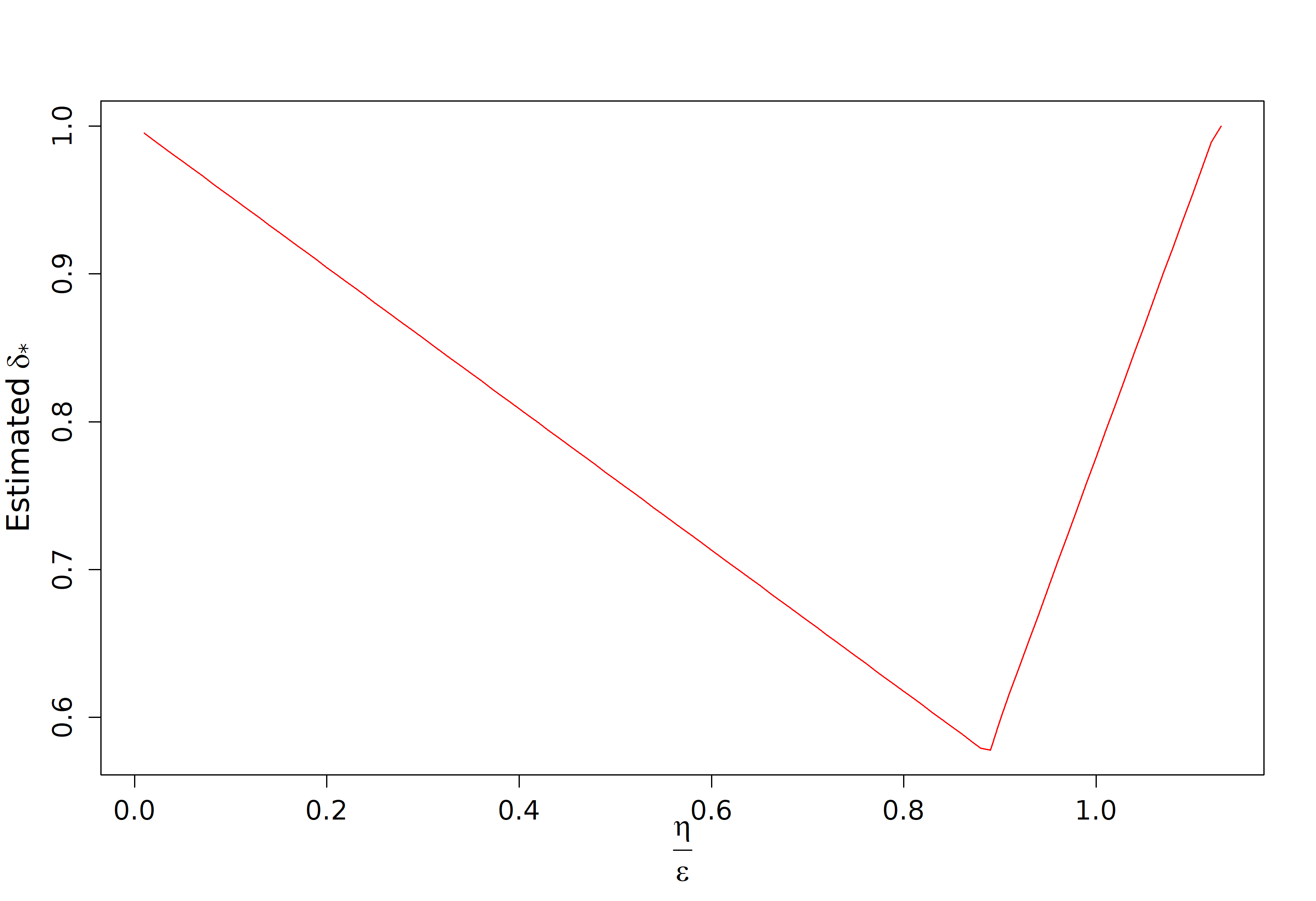}\\
    \includegraphics[width=0.49\linewidth]{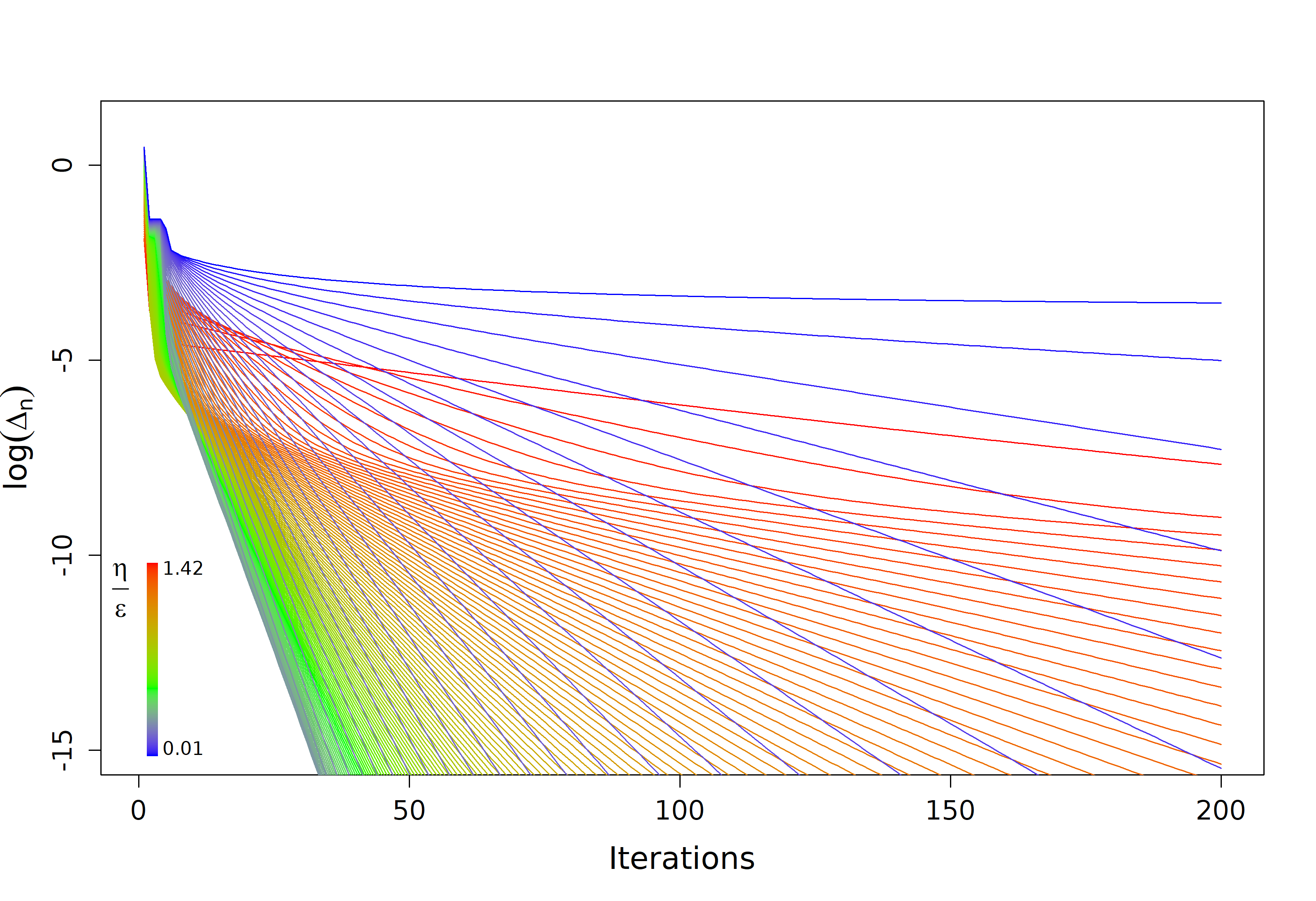}
    \includegraphics[width=0.49\linewidth]{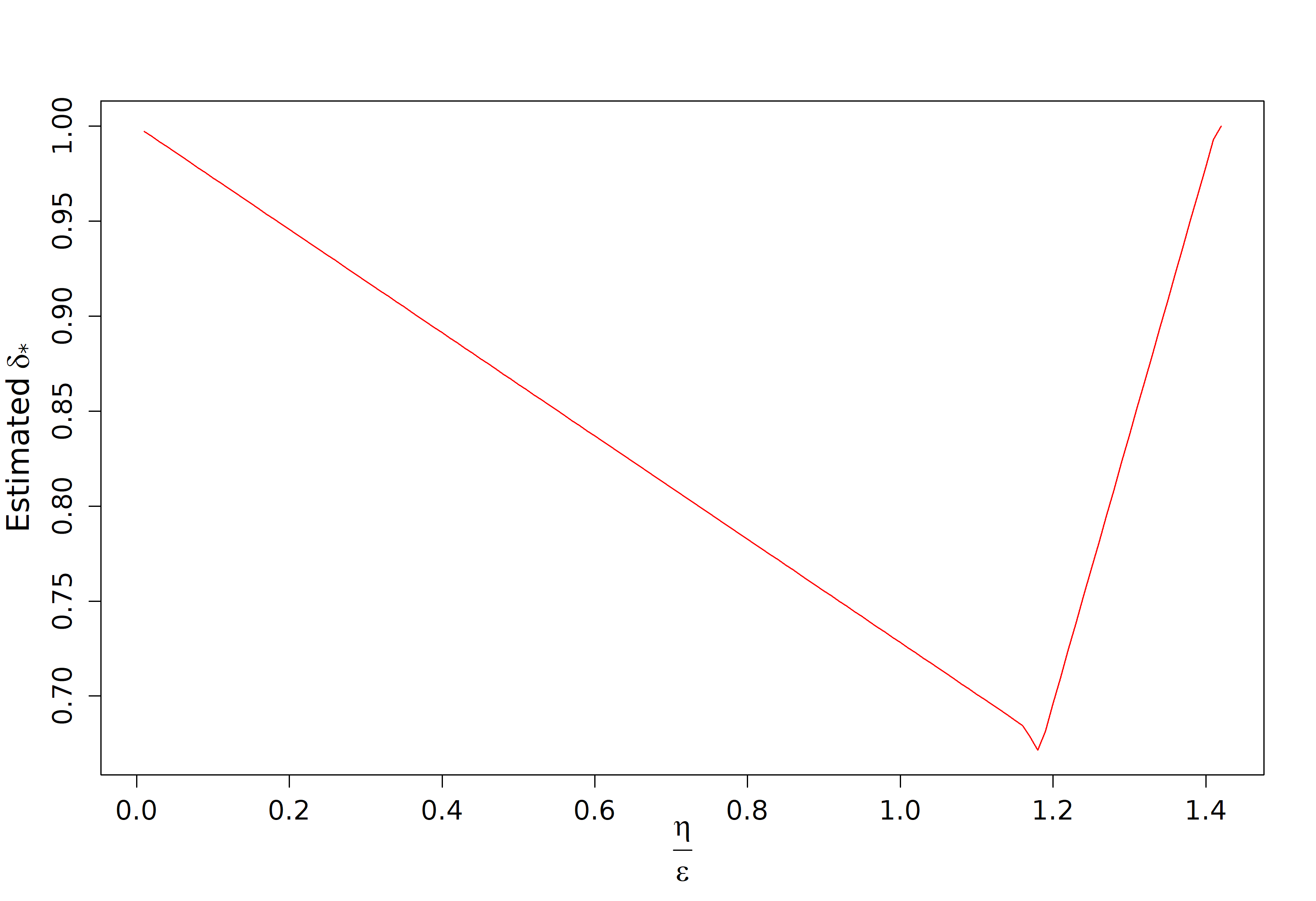}\\
    \includegraphics[width=0.49\linewidth]{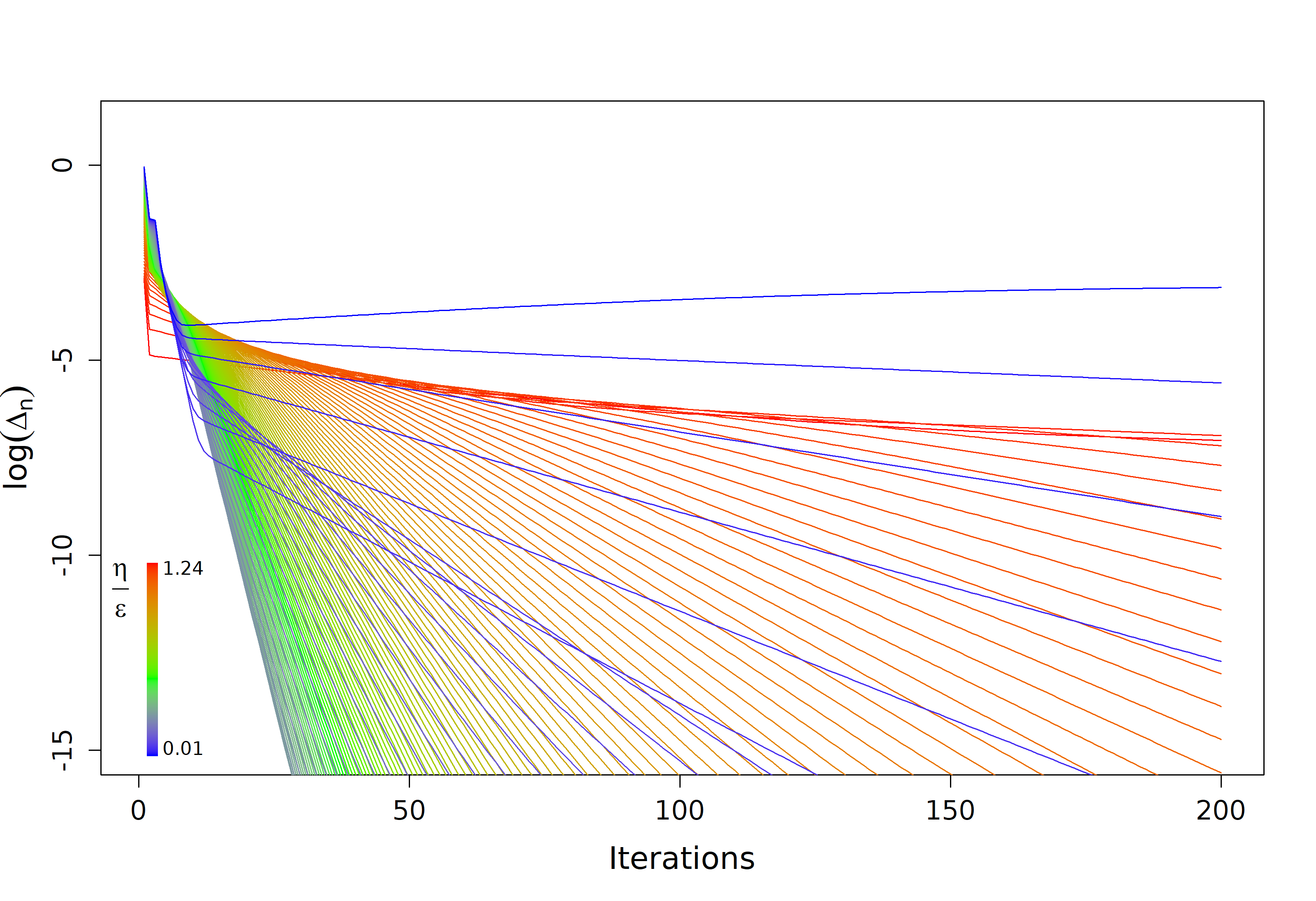}
    \includegraphics[width=0.49\linewidth]{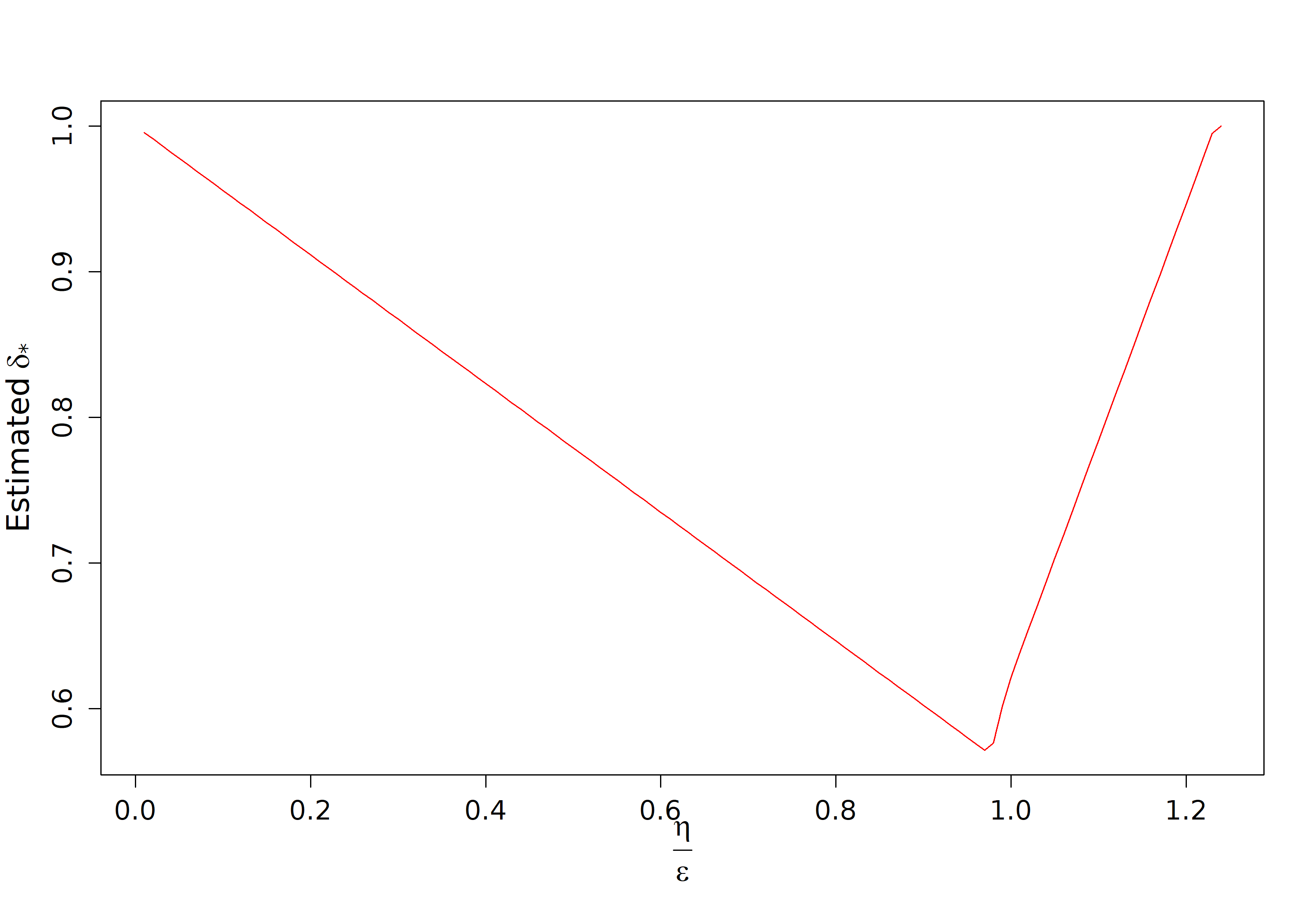}
    \caption{Gradient descent algorithm for three pairs of marginals: (top) $P = U[0,1]$ and $Q = U[0.5,1.5]$, (middle) $P = U[0,1]$ and $Q = \beta(0.1,0.2)$, (bottom) $P = U([0,1] \times [0,1])$ and $Q = U([1/\sqrt{2}, 1 + 1/\sqrt{2}] \times [1/\sqrt{2}, 1 + 1/\sqrt{2}])$. Left panels show the convergence for different step sizes $\eta$, right panels show the estimated $\delta_*$ for different $\eta$, both for fixed $\eps=10^{-1}$. Step size is increased until convergence breaks, which happens at $\eta/\eps= 1.13,\,1.42,\,1.24$, respectively.}
    \label{fig:experimentsDelta}
\end{figure}

The left panels in \cref{fig:experimentsDelta} plot $\log(\Delta_n)$ against~$n$ for three different pairs of marginals $(P,Q)$ and for numerous different step sizes $\eta$. 
The regularization parameter $\eps=10^{-1}$ and the initialization $f_{0} \equiv g_{0} \equiv 0.5$ are fixed. The gradient descent iteration is run until $\Delta_n\leq 10^{-10}$. The integrals in~\eqref{eq:gradientGamma} are approximated using the regular trapezoidal rule over a discretized mesh of $[-0.1, 1.6]$, with step size 0.001, giving approximately 1700 discretization points for the one-dimensional marginals, while the same integration procedure is applied on a mesh of $[-0.1,2] \times [-0.1, 2]$, with step size 0.05, to obtain approximately the same number of discretization points. 

For all step sizes $\eta <  \eps$, we observe a linear behavior after a burn-in period, confirming that $\Delta_n = O(\delta_*^n)$ where $\delta_*\in(0,1)$. We also show some step sizes with $\eta/\eps \geq 1$; specifically, we increase the step size until convergence breaks. In the three experiments, convergence breaks when $\eta/\eps$ reaches the values 1.13, 1.42, and 1.24. This suggests that the condition $\eta/\eps <  1$ in \cref{Assumptions:GD} is fairly sharp; see also \cref{rk:ComparisonWith2overL}.

The right panels in \cref{fig:experimentsDelta} show an estimate of $\delta_*$ for each step size. The estimate is obtained by a linear regression over $\log(\Delta_{N-k})$, $k=0,\dots,9$, where $N$ is the iteration where the convergence criterion is reached. We observe that $\delta_*$ depends substantially on $\eta/\eps$, with a distinct V-shape. 
Overall, the value of $\delta_*$ is quite small for a large range of step sizes (bounded away from zero and the break point), indicating fast convergence. As a caveat, we emphasize that the estimated~$\delta_*$ pertains to the particular trajectory of the algorithm which depends on the chosen initialization, meaning that the constant in \cref{Theorem:Explicit} could be worse.

\begin{figure}[h]
    \centering
    \includegraphics[width=0.495\linewidth]{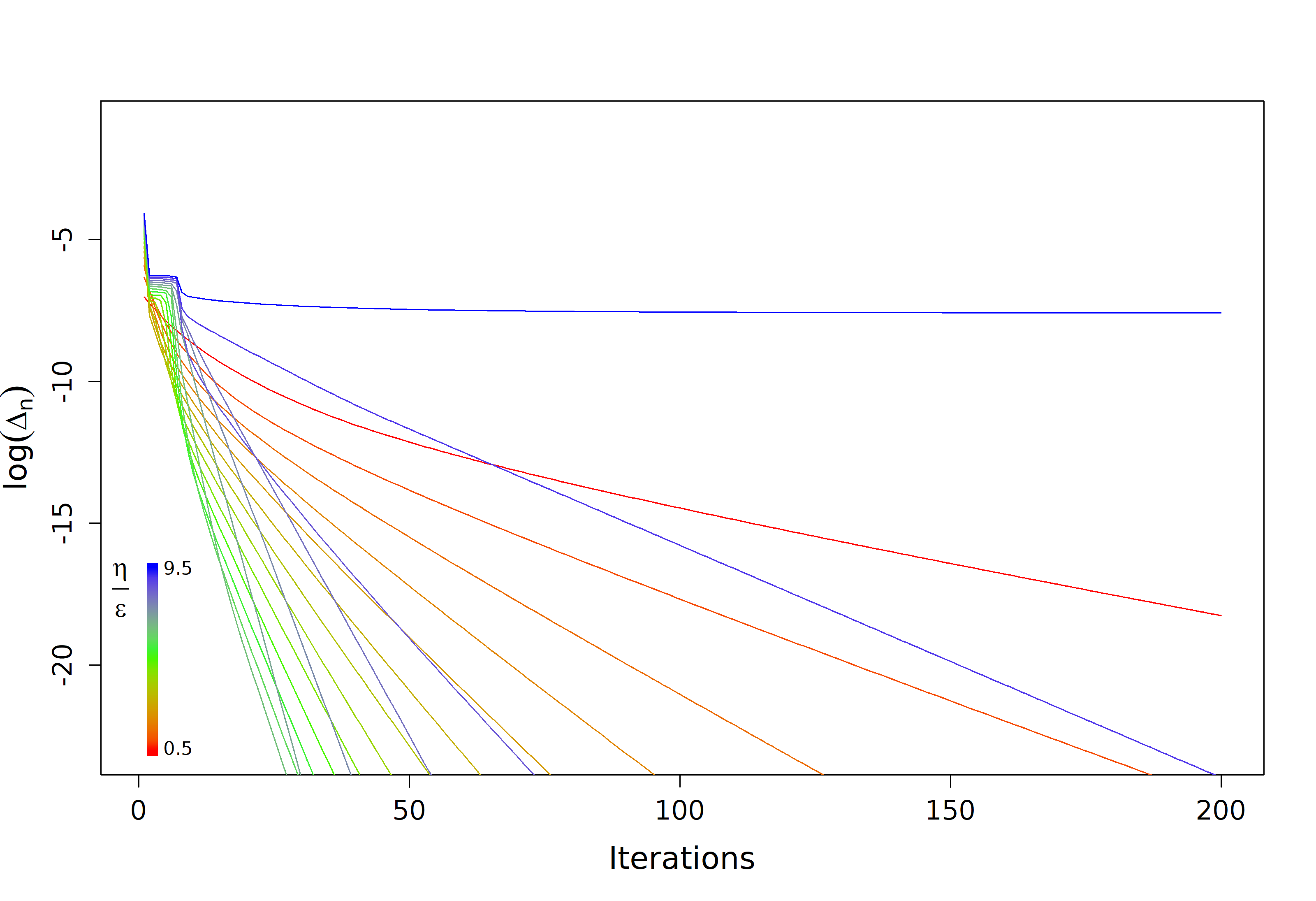}
    \includegraphics[width=0.495\linewidth]{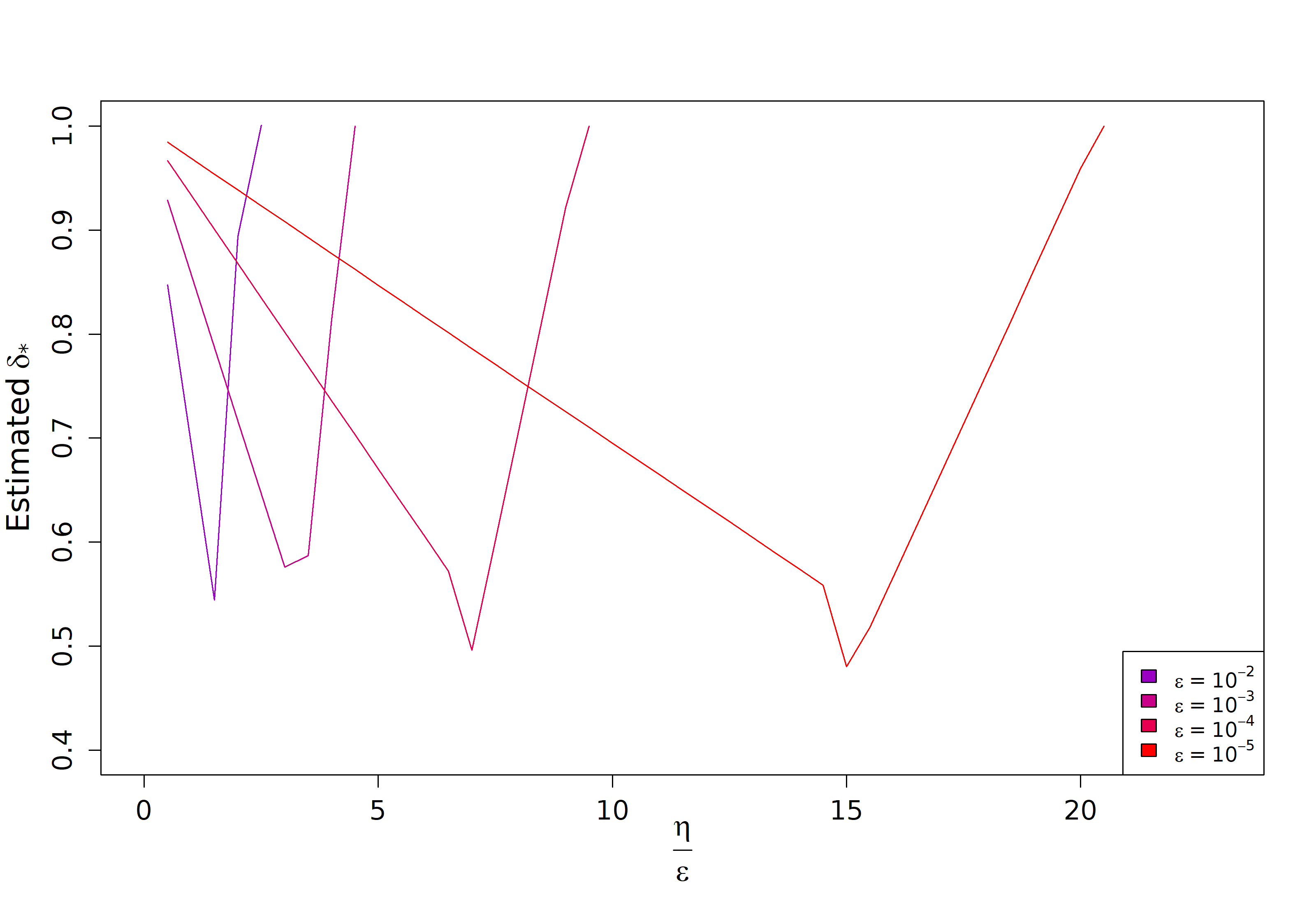}
    \caption{Repeating the first experiment (top row of \cref{fig:experimentsDelta}) with smaller values of $\eps$ and warm-start initialization. Left panel shows convergence for $\eps=10^{-4}$ and different step sizes~$\eta$. Right panel shows estimated $\delta_*$ for several values of $\eps$ and varying step size~$\eta$. Step size is increased until convergence breaks.}
    \label{fig:experimentsEpsilon_1}
\end{figure}

\Cref{fig:experimentsEpsilon_1} repeats the first experiment (i.e., the top row in \cref{fig:experimentsDelta}) but varies the regularization parameter $\eps$. In addition, the experiment for $\eps=10^{-k}$ is initialized with the potentials found for $\eps=10^{-k+1}$. Even so, a longer burn-in period is observed for smaller values of $\eps$ (for constant initialization, the burn-in period would be even longer). We observe that the break point as well as the optimal ratio $\eta/\eps$ increase as $\eps$ decreases, extending beyond \cref{Assumptions:GD}. Comparing with \cref{rk:ComparisonWith2overL}, a possible explanation is that for small values of $\eps$, the Lipschitz constant of the gradient ${\rm D}\Gamma$ is substantially smaller than $2/\eps$. Specifically, the positive part operator in~\eqref{eq:gradientGamma} implies that the integral is approximately taken only over the support of the optimal coupling. This support is conjectured to be sparse, with sections of diameter $\sim\eps^{\frac{1}{d+2}}$ (see \cite{GonzalezSanzNutz2024.Scalar,WieselXu.24}), which suggests a Lipschitz constant that shrinks with~$\eps$.

\paragraph{Acknowledgments} The authors thank Yifan Jiang for pointing out a mistake in an earlier version of the proof of \cref{lemma:consistency}.

\bibliographystyle{abbrv}
\bibliography{Biblio}

@article{combettes2018monotone,
  author  = {Combettes, Patrick L.},
  title   = {Monotone operator theory in convex optimization},
  journal = {Math. Program.},
  fjournal = {Mathematical Programming},
  series  = {Series B},
  volume  = {170},
  number  = {1},
  pages   = {177--206},
  year    = {2018},
  doi     = {10.1007/s10107-018-1303-3},
  url     = {https://doi.org/10.1007/s10107-018-1303-3}
}

@inproceedings{Pasechnyuk2023Algorithms,
  author    = {Pasechnyuk, Dmitry A. and Persiianov, Michael and Dvurechensky, Pavel and Gasnikov, Alexander},
  editor    = {Olenev, Nicholas and Evtushenko, Yuri and Ja{\'c}imovi{\'c}, Milojica and Khachay, Michael and Malkova, Vlasta},
  title     = {Algorithms for {E}uclidean-Regularised Optimal Transport},
  booktitle = {Optimization and Applications},
  series    = {Lecture Notes in Computer Science},
  volume    = {14395},
  pages     = {84--98},
  publisher = {Springer Nature Switzerland},
  address   = {Cham},
  year      = {2023},
  doi       = {10.1007/978-3-031-47859-8_7},
  isbn      = {978-3-031-47859-8}
}

@misc{gonzalezsanz2026localsparsity,
  author        = {Gonz{\'a}lez-Sanz, Alberto and Gvalani, Rishabh S. and Koch, Lukas},
  title         = {Sharp local sparsity of regularized optimal transport},
  year          = {2026},
  note        = {arXiv 2604.00843},
}

@article{GonzalezSanzNutz.26stability,
author={Alberto Gonz{\'a}lez-Sanz and Marcel Nutz},
title={Stability of Quadratically Regularized Optimal Transport},  
Journal = {In preparation},
year = {2026},
}

@article{GonzalezSanzNutzRiveros.26,
author={Alberto Gonz{\'a}lez-Sanz and Marcel Nutz and Riveros Valdevenito, Andr{\'e}s},
title={Polyak{--}{{\L{}}}ojasiewicz Inequality for Quadratically Regularized Optimal Transport},  
Journal = {In preparation},
year = {2026},
}

@unpublished{GonzalezSanzDelBarrioNutz.25,
  title     = {Sample Complexity of Quadratically Regularized Optimal Transport},
  author    = {Alberto Gonz{\'a}lez-Sanz and Eustasio del Barrio and Marcel Nutz},
  note      = {arXiv 2511.09807},
  year      = {2025},
}

@inproceedings{LiGenevayYurochkinSolomon.20,
 author = {Li, Lingxiao and Genevay, Aude and Yurochkin, Mikhail and Solomon, Justin},
 booktitle = {Advances in Neural Information Processing Systems},
 editor = {H. Larochelle and M. Ranzato and R. Hadsell and M.F. Balcan and H. Lin},
 pages = {17755--17765},
 publisher = {Curran Associates, Inc.},
 title = {Continuous Regularized {W}asserstein Barycenters},
 _url = {https://proceedings.neurips.cc/paper_files/paper/2020/file/cdf1035c34ec380218a8cc9a43d438f9-Paper.pdf},
 volume = {33},
 year = {2020}
}

@inproceedings{GeneveyEtAl.16,
title = {Stochastic Optimization for Large-scale Optimal Transport},
author = {Genevay, A. and Cuturi, M. and Peyr\'{e}, G. and Bach, F.},
booktitle = {Advances in Neural Information Processing Systems 29},
_editor = {D. D. Lee and M. Sugiyama and U. V. Luxburg and I. Guyon and R. Garnett},
pages = {3440--3448},
year = {2016},
_publisher = {Curran Associates, Inc.},
_url = {http://papers.nips.cc/paper/6566-stochastic-optimization-for-large-scale-optimal-transport.pdf}
}

@inproceedings{seguy2018large,
title={Large Scale Optimal Transport and Mapping Estimation},
author={Vivien Seguy and Bharath Bhushan Damodaran and Remi Flamary and Nicolas Courty and Antoine Rolet and Mathieu Blondel},
booktitle={International Conference on Learning Representations},
year={2018},
_url={https://openreview.net/forum?id=B1zlp1bRW},
}

@article {EcksteinKupper.21,
    AUTHOR = {Eckstein, Stephan and Kupper, Michael},
     TITLE = {Computation of optimal transport and related hedging problems
              via penalization and neural networks},
   JOURNAL = {Appl. Math. Optim.},
  FJOURNAL = {Applied Mathematics and Optimization},
    VOLUME = {83},
      YEAR = {2021},
    NUMBER = {2},
     PAGES = {639--667},
      _ISSN = {0095-4616,1432-0606},
   MRCLASS = {49Q22 (49M25 49M29 65K99 68T07)},
  MRNUMBER = {4239795},
MRREVIEWER = {Giuseppe\ Devillanova},
       _DOI = {10.1007/s00245-019-09558-1},
       _URL = {https://doi.org/10.1007/s00245-019-09558-1},
}

@inproceedings{GulrajaniAhmedArjovskyDumoulinCourville.17,
author = {Gulrajani, Ishaan and Ahmed, Faruk and Arjovsky, Martin and Dumoulin, Vincent and Courville, Aaron},
title = {Improved Training of {W}asserstein {GAN}s},
year = {2017},
_isbn = {9781510860964},
_publisher = {Curran Associates Inc.},
_address = {Red Hook, NY, USA},
booktitle = {Proceedings of the 31st International Conference on Neural Information Processing Systems},
pages = {5769--5779},
_series = {NIPS'17}
}

@article{WieselXu.24,
    AUTHOR = {Wiesel, Johannes and Xu, Xingyu},
     TITLE = {Sparsity of Quadratically Regularized Optimal
              Transport: Bounds on Concentration and Bias},
   JOURNAL = {SIAM J. Math. Anal.},
  FJOURNAL = {SIAM Journal on Mathematical Analysis},
    VOLUME = {57},
      YEAR = {2025},
    NUMBER = {6},
     PAGES = {6498--6521},
      ISSN = {0036-1410,1095-7154},
   MRCLASS = {49Q22},
  MRNUMBER = {4986738},
       DOI = {10.1137/25M1723633},
       URL = {https://doi.org/10.1137/25M1723633},
}

@article {EssidSolomon.18,
    AUTHOR = {Essid, Montacer and Solomon, Justin},
     TITLE = {Quadratically regularized optimal transport on graphs},
   JOURNAL = {SIAM J. Sci. Comput.},
  FJOURNAL = {SIAM Journal on Scientific Computing},
    VOLUME = {40},
      YEAR = {2018},
    NUMBER = {4},
     PAGES = {A1961--A1986},
      _ISSN = {1064-8275},
   MRCLASS = {65K10 (05C90 90C35)},
  MRNUMBER = {3820384},
MRREVIEWER = {Savin Trean\c{t}\u{a}},
       _DOI = {10.1137/17M1132665},
       _URL = {https://doi.org/10.1137/17M1132665},
}

@article{Muzellec.2017.AAAI,
  title = {Tsallis Regularized Optimal Transport and Ecological Inference},
  volume = {31},
  ISSN = {2159-5399},
  url = {http://dx.doi.org/10.1609/aaai.v31i1.10854},
  DOI = {10.1609/aaai.v31i1.10854},
  number = {1},
  journal = {Proceedings of the AAAI Conference on Artificial Intelligence},
  publisher = {Association for the Advancement of Artificial Intelligence (AAAI)},
  author = {Muzellec,  Boris and Nock,  Richard and Patrini,  Giorgio and Nielsen,  Frank},
  year = {2017},
  month = feb 
}

@article {Schmitzer.19,
    AUTHOR = {Schmitzer, B.},
     TITLE = {Stabilized sparse scaling algorithms for entropy regularized
              transport problems},
   JOURNAL = {SIAM J. Sci. Comput.},
  FJOURNAL = {SIAM Journal on Scientific Computing},
    VOLUME = {41},
      YEAR = {2019},
    NUMBER = {3},
     PAGES = {A1443--A1481},
      ISSN = {1064-8275},
   MRCLASS = {49Q20 (65K05 90C25)},
  MRNUMBER = {3947294},
       DOI = {10.1137/16M1106018},
       URL = {https://doi.org/10.1137/16M1106018},
}

@article{ChizatDelalandeVaskevicius.25,
  author  = {Chizat, L{\'e}na{\"\i}c and Delalande, Alex and Va{\v{s}}kevi{\v{c}}ius, Tomas},
  title   = {Sharper exponential convergence rates for {Sinkhorn}'s algorithm in continuous settings},
  journal = {Math. Program.},
  fjournal = {Mathematical Programming},
  series  = {Series A},
  volume  = {215},
  pages   = {809--858},
  year    = {2026},
  doi     = {10.1007/s10107-025-02242-z},
  url     = {https://doi.org/10.1007/s10107-025-02242-z}
}

@book {Nesterov.04,
    AUTHOR = {Nesterov, Yurii},
     TITLE = {Introductory lectures on convex optimization},
    SERIES = {Applied Optimization},
    VOLUME = {87},
      NOTE = {A basic course},
 PUBLISHER = {Kluwer Academic Publishers, Boston, MA},
      YEAR = {2004},
     PAGES = {xviii+236},
      ISBN = {1-4020-7553-7},
   MRCLASS = {90-02 (90-01 90C25)},
  MRNUMBER = {2142598},
       DOI = {10.1007/978-1-4419-8853-9},
       URL = {https://doi.org/10.1007/978-1-4419-8853-9},
}

@article {FRANKLIN.1989,
    AUTHOR = {Franklin, Joel and Lorenz, Jens},
     TITLE = {On the scaling of multidimensional matrices},
   JOURNAL = {Linear Algebra Appl.},
  FJOURNAL = {Linear Algebra and its Applications},
    VOLUME = {114/115},
      YEAR = {1989},
     PAGES = {717--735},
      ISSN = {0024-3795,1873-1856},
   MRCLASS = {15A99},
  MRNUMBER = {986904},
MRREVIEWER = {D.\ S.\ Tracy},
       DOI = {10.1016/0024-3795(89)90490-4},
       URL = {https://doi.org/10.1016/0024-3795(89)90490-4},
}

@article{LorenzMahler.22,
    AUTHOR = {Lorenz, Dirk and Mahler, Hinrich},
     TITLE = {Orlicz space regularization of continuous optimal transport
              problems},
   JOURNAL = {Appl. Math. Optim.},
  FJOURNAL = {Applied Mathematics and Optimization},
    VOLUME = {85},
      YEAR = {2022},
    NUMBER = {2},
     PAGES = {Paper No. 14, 33},
      _ISSN = {0095-4616,1432-0606},
   MRCLASS = {49Q22 (49J45 49N15)},
  MRNUMBER = {4409806},
MRREVIEWER = {Daniele\ Semola},
       _DOI = {10.1007/s00245-022-09826-7},
       _URL = {https://doi.org/10.1007/s00245-022-09826-7},
}

@InProceedings{blondel18quadratic,
  title = 	 {Smooth and Sparse Optimal Transport},
  author = 	 {Blondel, Mathieu and Seguy, Vivien and Rolet, Antoine},
  _booktitle = 	 {Proceedings of the Twenty-First International Conference on Artificial Intelligence and Statistics},
  pages = 	 {880--889},
  year = 	 {2018},
  _editor = 	 {Storkey, Amos and Perez-Cruz, Fernando},
  volume = 	 {84},
  series = 	 {Proceedings of Machine Learning Research},
  _month = 	 {09--11 Apr},
  _publisher =    {PMLR},
  _pdf = 	 {http://proceedings.mlr.press/v84/blondel18a/blondel18a.pdf},
  _url = 	 {https://proceedings.mlr.press/v84/blondel18a.html},
}

@article{GonzalezSanzNutz2024.Scalar,
      title={Sparsity of Quadratically Regularized Optimal Transport: Scalar Case}, 
      author={Alberto González-Sanz and Marcel Nutz},
      _year={2024},
     JOURNAL = {SIAM J. Math. Anal., forthcoming.},
     note={arXiv:2410.03353},
      archivePrefix={arXiv},
      primaryClass={math.OC}
}

@book {Brezis2011,
    AUTHOR = {Brezis, Haim},
     TITLE = {Functional analysis, {S}obolev spaces and partial differential
              equations},
    SERIES = {Universitext},
 PUBLISHER = {Springer, New York},
      YEAR = {2011},
     PAGES = {xiv+599},
      _ISBN = {978-0-387-70913-0}
}

@article{GarrizmolinaElAl.2024,
      title={Infinitesimal behavior of Quadratically Regularized Optimal Transport and its relation with the Porous Medium Equation}, 
      author={Alejandro Garriz-Molina and Alberto González-Sanz and Gilles Mordant},
      year={2024},
     journal={ arXiv:2407.21528},
      archivePrefix={arXiv},
      primaryClass={math.AP}
}

@article {Eckstein.Nutz.2023,
    AUTHOR = {Eckstein, Stephan and Nutz, Marcel},
     TITLE = {Convergence rates for regularized optimal transport via
              quantization},
   JOURNAL = {Math. Oper. Res.},
  FJOURNAL = {Mathematics of Operations Research},
    VOLUME = {49},
      YEAR = {2024},
    NUMBER = {2},
     PAGES = {1223--1240},
      ISSN = {0364-765X,1526-5471},
   MRCLASS = {49Q22},
  MRNUMBER = {4755769},
MRREVIEWER = {Emanuel\ Indrei},
       DOI = {10.1287/moor.2022.0245},
       URL = {https://doi.org/10.1287/moor.2022.0245},
}

@article{BayraktarEckstein.2025.BJ,
  title={Stability and sample complexity of divergence regularized optimal transport},
  author={Bayraktar, Erhan and Eckstein, Stephan and Zhang, Xin},
  journal={Bernoulli},
  volume={31},
  number={1},
  pages={213--239},
  year={2025},
  publisher={Bernoulli Society for Mathematical Statistics and Probability}
}

@article{zhang.2023.manifoldlearningsparseregularised,
      title={Manifold Learning with Sparse Regularised Optimal Transport}, 
      author={Stephen Zhang and Gilles Mordant and Tetsuya Matsumoto and Geoffrey Schiebinger},
      year={2023},
     journal={ arXiv:2307.09816},
      archivePrefix={arXiv},
      primaryClass={stat.ML}
}

@inproceedings{Lahn.Mulchandani.Raghvendra.2019.NEURIPS,
author = {Lahn, Nathaniel and Mulchandani, Deepika and Raghvendra, Sharath},
title = {A graph theoretic additive approximation of optimal transport},
year = {2019},
publisher = {Curran Associates, Inc.},
booktitle = {Proceedings of the 33rd International Conference on Neural Information Processing Systems},
articleno = {1239},
  pages     = {13831--13841},
}

@article {Carlier.2022.SIOPT,
    AUTHOR = {Carlier, Guillaume},
     TITLE = {On the linear convergence of the multimarginal {S}inkhorn
              algorithm},
   JOURNAL = {SIAM J. Optim.},
  FJOURNAL = {SIAM Journal on Optimization},
    VOLUME = {32},
      YEAR = {2022},
    NUMBER = {2},
     PAGES = {786--794},
      ISSN = {1052-6234,1095-7189},
   MRCLASS = {49Q22 (45G15 49M05 65R20)},
  MRNUMBER = {4418040},
MRREVIEWER = {M.\ Isabel\ Berenguer},
       DOI = {10.1137/21M1410634},
       URL = {https://doi.org/10.1137/21M1410634},
}

@inproceedings{Cuturi.2013.Neurips,
 author = {Cuturi, Marco},
 booktitle = {Advances in Neural Information Processing Systems},
 editor = {C.J. Burges and L. Bottou and M. Welling and Z. Ghahramani and K.Q. Weinberger},
 pages = {},
 publisher = {Curran Associates, Inc.},
 title = {Sinkhorn Distances: Lightspeed Computation of Optimal Transport},
 url = {https://proceedings.neurips.cc/paper_files/paper/2013/file/af21d0c97db2e27e13572cbf59eb343d-Paper.pdf},
 volume = {26},
 year = {2013}
}

@article {Nutz.2024,
    AUTHOR = {Nutz, Marcel},
     TITLE = {Quadratically regularized optimal transport: existence and
              multiplicity of potentials},
   JOURNAL = {SIAM J. Math. Anal.},
  FJOURNAL = {SIAM Journal on Mathematical Analysis},
    VOLUME = {57},
      YEAR = {2025},
    NUMBER = {3},
     PAGES = {2622--2649},
      ISSN = {0036-1410,1095-7154},
   MRCLASS = {49Q22 (90C25)},
  MRNUMBER = {4907548},
       DOI = {10.1137/24M1638136},
       URL = {https://doi.org/10.1137/24M1638136},
}

@article {Lorenz.2019,
    AUTHOR = {Lorenz, Dirk A. and Manns, Paul and Meyer, Christian},
     TITLE = {Quadratically regularized optimal transport},
   JOURNAL = {Appl. Math. Optim.},
  FJOURNAL = {Applied Mathematics and Optimization},
    VOLUME = {83},
      YEAR = {2021},
    NUMBER = {3},
     PAGES = {1919--1949},
      ISSN = {0095-4616,1432-0606},
   MRCLASS = {49Q22 (49N15 65H99 90C25)},
  MRNUMBER = {4261277},
MRREVIEWER = {Luca\ Granieri},
       DOI = {10.1007/s00245-019-09614-w},
       URL = {https://doi.org/10.1007/s00245-019-09614-w},
}

@article{Peyre.Cuturi.2019.Book,
  title = {Computational Optimal Transport: With Applications to Data Science},
  volume = {11},
  ISSN = {1935-8245},
  url = {http://dx.doi.org/10.1561/2200000073},
  DOI = {10.1561/2200000073},
  number = {5–6},
  journal = {Foundations and Trends{\textregistered} in Machine Learning},
  publisher = {Now Publishers},
  author = {Peyré,  Gabriel and Cuturi,  Marco},
  year = {2019},
  pages = {355–607}
}

@article{lorenz.2019.preprint,
      title={Orlicz-space regularization for optimal transport and algorithms for quadratic regularization}, 
      author={Dirk A. Lorenz and Hinrich Mahler},
      year={2019},
      journal={arXiv:1909.06082}
}
\end{document}